\documentclass[reqno,english]{amsart}
\usepackage{amsfonts,amsmath,latexsym,verbatim,amscd,mathrsfs,color,array}
\usepackage[colorlinks=true]{hyperref}

\usepackage{amsmath,amssymb,amsthm,amsfonts,graphicx,color}
\usepackage{amssymb}
\usepackage{pdfsync}
\usepackage{epstopdf}

\usepackage{cite}

 \usepackage{graphicx}

\providecommand{\U}[1]{\protect\rule{.1in}{.1in}}
\newtheorem{theorem}{Theorem}

\newtheorem{lemma}[theorem]{Lemma}

\newtheorem{proposition}[theorem]{Proposition}
\newtheorem{remark}[theorem]{Remark}

\numberwithin{equation}{section}

\begin{document}

\title[KP-I Lump Solution]{Nondegeneracy, Morse index and orbital stability of the KP-I lump solution }

\author[Y.Liu]{Yong Liu}

\address{\noindent School of Mathematics and Physics, North China Electric Power University, Beijing, China}
\email{liuyong@ncepu.edu.cn}

\author[J. Wei]{Juncheng Wei}
\address{\noindent
Department of Mathematics,
University of British Columbia, Vancouver, B.C., Canada, V6T 1Z2}
\email{jcwei@math.ubc.ca}

\begin{abstract}
Using B\"{a}cklund
transformation the spectral property of the KP-I lump solution is completely analyzed. It is proved that the lump solution is nondegenerate and has Morse
index one. As a consequence  we show that it is orbitally stable.

\end{abstract}

\maketitle

\section{Introduction and statement of the main results}

The KP equation, introduced by Kadomtsev and Petviashvili \cite{KP}, is a
classical nonlinear dispersive equation appearing in many physical contexts,
including the motion of shallow water waves. It has the form
\begin{equation}
\partial_{x}\left(  \partial_{t}u+\partial_{x}^{3}u+3\partial_{x}\left(
u^{2}\right)  \right)  -\sigma\partial_{y}^{2}u=0, \label{KPI}%
\end{equation}
where $\sigma$ is a parameter, usually chosen to be $\pm1.$ In this paper, we
will consider the case of $\sigma=1,$ which is called KP-I equation. It is
worth pointing out that the KP-II equation, corresponding to $\sigma=-1$, has
different dispersion relation and its property is quite different from that of KP-I.

If the function $u$ is $y$ independent, then equation $\left(  \ref{KPI}%
\right)  $ will be reduced to the KdV equation, a well known integrable system which has
been extensively studied for about half a century. Dryuma \cite{Dry} in 1974
found a Lax pair for the KP equation. The inverse scattering transform for
KP-I equation then has been carried out in \cite{Fokas, Mana1,Zhou}. (See
\cite{Ablowitz1, Ablowitz2} and the references therein for more
discussions on the results concerning this topic obtained before 1991.)

The soliton solutions of the KP equation can be obtained through various
different methods, both for the KP-I and KP-II equations. An important feature
of the KP-I equation is that it admits a family of lump solutions which are
also of travelling wave type. They were found in \cite{Manakov, Sa}.
Explicitly, the KP-I equation has a family of lump solutions of the form%
\begin{equation}
u_{c}=4\frac{-\left(  x-ct\right)  ^{2}+cy^{2}+\frac{3}{c}}{\left(  \left(
x-ct\right)  ^{2}+cy^{2}+\frac{3}{c}\right)  ^{2}}.
\end{equation}
It particular, for $c=1,$ it can be written as $Q\left(  x-t,y\right)  ,$
where
\begin{equation}
Q\left(  x,y\right)  =4\frac{y^{2}-x^{2}+3}{\left(  x^{2}+y^{2}+3\right)
^{2}}.
\end{equation}
Note that $Q$ is non-radial and decays like $O\left(  \frac{1}{x^{2}+y^{2}}\right)  $ at
infinity. Because of this slow decaying property, the inverse scattering
transform of KP-I turns out to be quite delicate. Indeed, it is shown in
\cite{VA} that a winding number can be associated to the lump solutions. We
also mention that the interaction of multi-lump solutions has been
investigated in \cite{G,Lu}.

The KP-I equation and its lump solutions actually appear in many other
physical models. For instance, formal asymptotic analysis shows that in the
context of the motion in a Bose condensate, the transonic limit of certain
traveling wave solutions to the GP equation is related to the lump solutions
(\cite{Roberts,Roberts3}). This is the so called Roberts's Program
(\cite{Roberts2,Roberts3}). A rigorous verification is given recently in
\cite{BGS1} and \cite{CM}. The KP-I equation has other travelling wave
solutions with energy higher than the lump $Q$. Indeed, using binary Darboux
transformation and other methods, families of rational type solutions have
been studied in \cite{V2,AV,G,Pelin1,Pelin2,P1} (for instance, see P. 139 of
\cite{V2}). Explicitly, one of the families of these solutions has the form of
$2\partial_{x}^{2}\ln F,$ with%
\[
F=\left\vert f^{3}-4f+g\right\vert ^{2}+27\left\vert f^{2}-\frac{8}{3}\sqrt
{3}-3\right\vert ^{2}+162\left\vert f\right\vert ^{2}+1593,
\]
where $f=x+yi+\gamma$ and $g=12yi+\delta$, with $\gamma,\delta$ being complex
parameters. In particular, if $\gamma=\delta=0,$ then we get a solution with
$F$ being even in both $x$ and $y$ variables.

Regarding to traveling wave type solutions, it is worth mentioning that the
generalized KP-I equations
\begin{equation}
\partial_{x}^{2}\left(  \partial_{x}^{2}u-u+u^{p}\right)  -\partial_{y}^{2}u=0
\label{KPgeneralized}%
\end{equation}
also have lump type solutions for suitable $p\in\left(  1,5\right)
$(\cite{Saut,Liu}). Note that for general $p\neq2,$ this is believed to be not
an integrable system. Hence no explicit formula is available. These lump type
solutions are obtained via variational methods (concentration compactness). The
stability(or instability) and other properties of these ground state solutions
have been studied in \cite{Saut,Saut2,Saut3,Klein,LiuYue,Sau0,Sau,Wang}. It is
known that for $p\in\left(  1,\frac{7}{3}\right)  $, they are orbitally
stable, while for $p\in\left(  \frac{7}{3},5\right)  $, they are unstable.
However, it is not known whether or not  $Q$ has this variational characterization. It
is conjectured that for $p=2,$ there should be a unique (up to translation)
ground state, which should be the lump solution $Q.$ We refer to the very
interesting paper of \cite{Klein} for a review and numerical study of the lumps.

\medskip

Our first result in this paper is

\begin{theorem}
\label{main} Suppose $\phi$ is a smooth solution to the equation
\begin{equation}
\partial_{x}^{2}\left(  \partial_{x}^{2}\phi-\phi+6Q\phi\right)  -\partial
_{y}^{2}\phi=0. \label{FI}%
\end{equation}
Assume that
\[
\phi\left(  x,y\right)  \rightarrow0,\text{ as }x^{2}+y^{2}\rightarrow
+\infty.
\]
Then $\phi=c_{1}\partial_{x}Q+c_{2}\partial_{y}Q,$ for some constants
$c_{1},c_{2}.$
\end{theorem}

Theorem \ref{main} has long been conjectured to be true. See the remark after
Lemma 7 of \cite{Wang}, concerning the spectral property and its relation to stability.

Our second result concerns the Morse index of $Q.$ By definition, the Morse
index of $Q$ is the number of negative eigenvalues (counted with multiplicity)
of the operator
\[
\mathcal{L}\phi:=-\partial_{x}^{2}\phi+\phi-6Q\phi+\partial_{x}^{-2}%
\partial_{y}^{2}\phi.
\]

\begin{theorem}
\label{thm2}
The operator $\mathcal{L}$ has exactly one negative eigenvalue and hence has
Morse index one.
\end{theorem}

With the spectral information of the lump at hand, it is natural to
investigate the nonlinear stability of the lump solution. By a result of
Ionescu-Kenig-Tataru \cite{Kenig2}, the KP-I equation is well-posed in the
natural energy space. Note that while the initial value problem of KP-II is
relatively easy to handle (see \cite{Bourgain}), the well-posedness (or
ill-posedness) of the KP-I equation is much more delicate. We also refer to
\cite{Kenig1,S1,M1,M2} and the references therein for more detailed discussion
on this topic. We prove in this paper:

\begin{theorem}
\label{thm3}
The lump solution  $u_{c}$ is orbitally stable.
\end{theorem}

We refer to Theorem \ref{stability} in Section \ref{Morse} for a more
precisely statement of this result. We note that
the Morse index of the lump solution is numerically shown to be one (\cite{CS}). Here we give a rigorous proof. It is also worth mentioning that the stability
or instability of the line solitons of the KP-I or KP-II equation has already
been studied. The instability of the line soliton of the KP-I equation has
been proved in \cite{Zah}. The spectral instability of the line solitons for
the generalized KP-I equation is proved in \cite{A}. In \cite{F0,F}, the
orbital stability or instability of line solitons of KP-I equation in
$y$-periodic space is studied. It is shown there that the stability is
determined by the travelling speed of the line solitons. The case of critical
speed is then investigated in a recent paper\cite{Yamakazi}. The orbital and
asymptotic stability of the line solitons for the KP-II equation is studied by
Mizumachi and Tzvetkov in \cite{Mizu3,Mizu5}. We also refer to \cite{Martel}
for a review on stability of solitons for the gKdV equations.

As we mentioned before, the KP-I equation has close relation with travelling
waves of the GP equation. For solutions of the form $\Phi\left(
z-ct,x\right)  ,$ the GP equation is
\begin{equation}
ic\partial_{z}\Phi=\Delta_{\left(  z,x\right)  }\Phi-\Phi\left(  \left\vert
\Phi\right\vert ^{2}-1\right)  ,\text{ }\left(  z,x\right)  \in\mathbb{R}^{2}.
\label{GP!}%
\end{equation}
Formal computation performed in the appendix of \cite{Roberts} suggests that
this equation has a family of solutions with $c$ close to $\sqrt{2}$, and
their asymptotic profiles are determined by the lump solution. Our
nondegeneracy results in this paper will have potential applications in the
rigorous construction of these solutions. For more discussions on the
traveling waves of Gross-Pitaevskii equation, we refer to
\cite{BGS1,BGS2,CS,Maris} and the references therein.

Now let us briefly describe the main ideas of the proof of our main results.
One of our tools will be the B\"{a}cklund transformation. We first show that
the lump solution can be obtained from the trivial solution by performing
B\"{a}cklund transformation twice. Note that while this fact should be known
to experts in the field, we have not been able to find a reference for this
result. Next we consider these transformations in the linearized level. We
show that a kernel of the linearized operator around $Q$ can be transformed to
a kernel of the operator $\partial_{x}^{2}+\partial_{y}^{2}-\partial_{x}^{4},$
which is the linearized operator around the trivial solution. With some
information on the growth rate of the kernel function, we are able to conclude
that the only decaying solutions to $\left(  \ref{FI}\right)  $ are
corresponding to translations in the $x$ and $y$ axes. This proves Theorem \ref{main}.  Note that the general
idea of using linearized B\"{a}cklund transformation to investigate the
spectral property has already been used in \cite{Mizu} in the case of Toda
lattice. We also refer to \cite{Mizu1,Mizu4} for related discussion about
$n$-soliton solutions for the KdV equation and the NLS equation.

To prove Theorem 2, we consider
the nondegeneracy of a family of $y$-periodic traveling wave solutions to KP-I connecting the one dimensional soliton and the two-dimensional lump solution.  They are given as follows: let $ k, b \geq 0$ with $ k^2+ b^2=1$. Define
$$ \Gamma_k= \mbox{cosh} ( k x)+\sqrt{\frac{1-4k^2}{1-k^2}} \mbox{cosh} ( k b i y), $$
\begin{equation}
 Q_k (x, y)= 2 \partial^2_x \mbox{ln} \Gamma_k.
 \end{equation}

Then $Q_k (x-t, y)$ are traveling wave solutions to KP-I (with speed $c=1$). They are periodic in $y$ with period $ t_k=\frac{2\pi}{k \sqrt{1-k^2}}$. As $ k\to 0$, $Q_k$ converges to the lump solution. As $ k \to \frac{1}{2}$, $ Q_k$ converges to the one-dimensional soliton solution $ \frac{1}{2} \mbox{sech}^2(\frac{x}{2})$.

\begin{theorem}
\label{thm4} The solution $Q_k$  is
nondegenerate in the following sense: Suppose $\varphi$ is a function
decaying in the $x$ variable and $t_k$-periodic in the $y$
variable, satisfying
\[
\partial_{x}^{2}\left(  \partial_{x}^{2}\varphi-\varphi+6 Q_k \varphi\right)
-\partial_{y}^{2}\varphi=0.
\]
Then $\varphi=c_{1}\partial_{x} Q_k +c_{2}\partial_{y} Q_k.$
\end{theorem}

With the help of Theorem \ref{thm4} we prove Theorem \ref{thm2} by a continuation argument.  By nondegeneracy along the path $ k \in [0, \frac{1}{2}]$,  the Morse index is invariant along the periodic solution. Since the Morse index of the one dimensional soliton is $1$, we conclude that the Morse index of the lump is also $1$.

Finally to prove Theorem \ref{thm3} we show the convexity of the energy
\[
d\left(  c\right)  :=\int_{\mathbb{R}^{2}}\left(  \frac{1}{2}\left(
\partial_{x}u_{c}\right)  ^{2}-u_{c}^{3}+\frac{1}{2}\left(  \partial_{y}\partial_{x}^{-1}u_{c}\right)  ^{2}+\frac{1}{2}cu_{c}^{2}\right)  .
\]
and  the classical result of Grillakis-Shatah-Strauss \cite{GSS} yields Theorem \ref{thm3}.

This paper is organized in the following way. In Section \ref{Sec2}, we recall
some basic facts about the bilinear derivatives. In Section \ref{lump}, we
prove the nondegenracy of lump. In Section \ref{period}, we prove the
nondegeneracy of a family of periodic solutions naturally associated to the
lump solution. In Section \ref{Morse}, we prove the orbital stability of lump.

\textit{Acknowledgement.} The research of J. Wei is partially supported by
NSERC of Canada. Part of this work is finished while the first author is
visiting the University of British Columbia in 2017. He thanks the institute
for the financial support. We thank Prof. J. C. Saut for pointing out to us
some key references.

\section{Preliminaries\label{Sec2}}

In this section we describe the bilinear form of KP-I and its associated B\"{a}cklund transform. 

We will use $D$ to denote the bilinear derivative. Explicitly,
\[
D_{x}^{m}D_{t}^{n}f\cdot g=\partial_{y}^{m}\partial_{s}^{n}f\left(
x+y,t+s\right)  g\left(  x-y,t-s\right)  |_{s=0,y=0.}%
\]
In particular,
\begin{align*}
D_{x}f\cdot g  &  =\partial_{x}fg-f\partial_{x}g,\\
D_{x}^{2}f\cdot g  &  =\partial_{x}^{2}fg-2\partial_{x}f\partial
_{x}g+f\partial_{x}^{2}g,\\
D_{x}D_{t}f\cdot g  &  =\partial_{x}\partial_{t}fg-\partial_{x}f\partial
_{t}g-\partial_{t}f\partial_{x}g+f\partial_{x}\partial_{t}g.
\end{align*}
The KP-I equation $\left(  \ref{KPI}\right)  $ can be written in the following
bilinear form \cite{H}:%
\begin{equation}
\left(  D_{x}D_{t}+D_{x}^{4}-D_{y}^{2}\right)  \tau\cdot\tau=0. \label{bi}%
\end{equation}
To explain this, we introduce the $\tau$-function:
\begin{equation}
u=2\partial_{x}^{2}\left(  \ln\tau\right)  . \label{t1}%
\end{equation}
Then equation $\left(  \ref{KPI}\right)  $ becomes
\[
\partial_{x}^{2}\left[  \left(  \partial_{x}\partial_{t}\ln\tau+\partial
_{x}^{4}\ln\tau+6\left(  \partial_{x}^{2}\ln\tau\right)  ^{2}\right)
-\partial_{y}^{2}\ln\tau\right]  =0.
\]
Using the identities (see Section 1.7.2 of \cite{H})
\[
2\partial_{x}^{2}\ln\tau=\frac{D_{x}^{2}\tau\cdot\tau}{\tau^{2}},\text{
}2\partial_{x}\partial_{t}\ln\tau=\frac{D_{x}D_{t}\tau\cdot\tau}{\tau^{2}},
\]
and
\[
2\partial_{x}^{4}\ln\tau=\frac{D_{x}^{4}\tau\cdot\tau}{\tau^{2}}-3\left(
\frac{D_{x}^{2}\tau\cdot\tau}{\tau^{2}}\right)  ^{2},
\]
we get
\[
\partial_{x}^{2}\left(  \frac{D_{x}D_{t}\tau\cdot\tau+D_{x}^{4}\tau\cdot
\tau-D_{y}^{2}\tau\cdot\tau}{\tau^{2}}\right)  =0.
\]
Therefore, if $\tau$ satisfies the bilinear equation $\left(  \ref{bi}\right)
,$ then $u$ satisfies the KP-I equation $\left(  \ref{KPI}\right)  .$

Let $i$ be the imaginary unit. We will investigate properties of the following
three functions:
\begin{align*}
\tau_{0}\left(  x,y\right)   &  =1,\\
\tau_{1}\left(  x,y\right)   &  =x+iy+\sqrt{3},\\
\tau_{2}\left(  x,y\right)   &  =x^{2}+y^{2}+3.
\end{align*}
Let
\begin{align*}
\tilde{\tau}_{0}\left(  x,y,t\right)   &  =\tau_{0}\left(  x-t,y\right)  ,\\
\tilde{\tau}_{1}\left(  x,y,t\right)   &  =\tau_{1}\left(  x-t,y\right)  ,\\
\tilde{\tau}_{2}\left(  x,y,t\right)   &  =\tau_{2}\left(  x-t,y\right)  .
\end{align*}
Then $\tilde{\tau}_{0},\tilde{\tau}_{1},\tilde{\tau}_{2}$ are solutions of
$\left(  \ref{bi}\right)  .$

\subsection{The B\"acklund transformation}

Under the transformation $\left(  \ref{t1}\right)  ,$ the function
$\tilde{\tau}_{2}$ is corresponding to the lump solution $Q\left(
x-t,y\right)  .$ Note that the solution corresponding to $\tilde{\tau}_{1}$ is
not real valued.

Now let $\mu$ be a constant. We first recall the following bilinear operator
identity (see \cite{Naka} or P. 90 in \cite{Rogers} for more details, we also
refer to \cite{Hu,Ma} and the references therein for the construction of more
general rational solutions, for KP and other related equations). Let $\mu
,\nu,\lambda$ be arbitrary parameters. Then we have%
\begin{align}
&  \frac{1}{2}\left[  \left(  D_{x}D_{t}+D_{x}^{4}-D_{y}^{2}\right)  f\cdot
f\right]  gg-\frac{1}{2}\left[  \left(  D_{x}D_{t}+D_{x}^{4}-D_{y}^{2}\right)
g\cdot g\right]  ff\nonumber\\
&  =D_{x}\left[  \left(  D_{t}+3\lambda D_{x}-\sqrt{3}i\mu D_{y}+D_{x}%
^{3}-\sqrt{3}iD_{x}D_{y}+\nu\right)  f\cdot g\right]  \cdot\left(  fg\right)
\nonumber\\
&  +3D_{x}\left[  \left(  D_{x}^{2}+\mu D_{x}+\frac{1}{\sqrt{3}}iD_{y}%
-\lambda\right)  f\cdot g\right]  \cdot\left(  D_{x}g\cdot f\right)
\nonumber\\
&  +\sqrt{3}iD_{y}\left[  \left(  D_{x}^{2}+\mu D_{x}+\frac{1}{\sqrt{3}}%
iD_{y}-\lambda\right)  f\cdot g\right]  \cdot\left(  fg\right)  . \label{back}%
\end{align}
By this identity, we can consider the B\"{a}cklund transformation from
$\tilde{\tau}_{0}$ to $\tilde{\tau}_{1}$($\mu=\frac{1}{\sqrt{3}}$):
\begin{equation}
\left\{
\begin{array}
[c]{l}%
\left(  D_{x}^{2}+\frac{1}{\sqrt{3}}D_{x}+\frac{1}{\sqrt{3}}iD_{y}\right)
\tilde{\tau}_{0}\cdot\tilde{\tau}_{1}=0,\\
\left(  D_{t}-iD_{y}+D_{x}^{3}-\sqrt{3}iD_{x}D_{y}\right)  \tilde{\tau}%
_{0}\cdot\tilde{\tau}_{1}=0.
\end{array}
\right.  \label{b1}%
\end{equation}
The B\"{a}cklund transformation from $\tilde{\tau}_{1}$ to $\tilde{\tau}_{2}$
is given by$\left(  \mu=-\frac{1}{\sqrt{3}}\right)  $
\begin{equation}
\left\{
\begin{array}
[c]{l}%
\left(  D_{x}^{2}-\frac{1}{\sqrt{3}}D_{x}+\frac{1}{\sqrt{3}}iD_{y}\right)
\tilde{\tau}_{1}\cdot\tilde{\tau}_{2}=0,\\
\left(  D_{t}+iD_{y}+D_{x}^{3}-\sqrt{3}iD_{x}D_{y}\right)  \tilde{\tau}%
_{1}\cdot\tilde{\tau}_{2}=0.
\end{array}
\right.  \label{b2}%
\end{equation}
Throughout the paper, we set $r=\sqrt{x^{2}+y^{2}}.$ We would like to relate
the kernel of the linearized KP-I equation to that of the linearized equation
in bilinear form.

\begin{lemma}
\label{core}Suppose $\phi$ is a smooth function satisfying the linearized KP-I
equation%
\[
\partial_{x}^{2}\left(  \partial_{x}^{2}\phi-\phi+6Q\phi\right)  -\partial
_{y}^{2}\phi=0.
\]
Assume
\[
\phi\left(  x,y\right)  \rightarrow0,\text{ as }r\rightarrow+\infty.
\]
Let
\[
\eta\left(  x,y\right)  =\tau_{2}\int_{0}^{x}\int_{-\infty}^{t}\phi\left(
s,y\right)  dsdt.
\]
Then $\eta$ satisfies
\[
\left(  -D_{x}^{2}+D_{x}^{4}-D_{y}^{2}\right)  \eta\cdot\tau_{2}=0.
\]
Moreover,
\begin{equation}
\left\vert \eta\right\vert +\left(  \left\vert \partial_{x}\eta\right\vert
+\left\vert \partial_{y}\eta\right\vert +\left\vert \partial_{x}^{2}%
\eta\right\vert +\left\vert \partial_{x}\partial_{y}\eta\right\vert
+\left\vert \partial_{x}^{3}\eta\right\vert \right)  \left(  1+r\right)  \leq
C\left(  1+r\right)  ^{\frac{5}{2}}. \label{es1}%
\end{equation}

\end{lemma}

\begin{proof}
We write the equation$\ $
\[
\partial_{x}^{2}\left(  \partial_{x}^{2}\phi-\phi+6Q\phi\right)  -\partial
_{y}^{2}\phi=0
\]
as
\[
\partial_{x}^{4}\phi-\partial_{x}^{2}\phi-\partial_{y}^{2}\phi=-6\partial
_{x}^{2}\left(  Q\phi\right)  .
\]
Since $\phi\rightarrow0$ as $r\rightarrow+\infty,$ it follows from the a
priori estimate of the operator $\partial_{x}^{2}-\partial_{x}^{4}%
+\partial_{y}^{2}$(see Lemma 3.6 of \cite{Saut2}) that
\[
\left\vert \phi\right\vert +\left(  \left\vert \partial_{x}\phi\right\vert
+\left\vert \partial_{y}\phi\right\vert \right)  \left(  1+r\right)  \leq
C\left(  1+r\right)  ^{-2}.
\]
Moreover, $\int_{-\infty}^{+\infty}\phi\left(  x,y\right)  dx=0,$ for all $y.$
Therefore,
\[
\left\vert \eta\right\vert +\left(  \left\vert \partial_{x}\eta\right\vert
+\left\vert \partial_{y}\eta\right\vert +\left\vert \partial_{x}^{2}%
\eta\right\vert +\left\vert \partial_{x}\partial_{y}\eta\right\vert
+\left\vert \partial_{x}^{3}\eta\right\vert \right)  \left(  1+r\right)  \leq
C\left(  1+r\right)  ^{\frac{5}{2}}.
\]
As a consequence,
\begin{equation}
\frac{-D_{x}^{2}\eta\cdot\tau_{2}+D_{x}^{4}\eta\cdot\tau_{2}-D_{y}^{2}%
\eta\cdot\tau_{2}}{\tau_{2}^{2}}\rightarrow0,\text{ as }r\rightarrow+\infty.
\label{g2}%
\end{equation}

On the other hand, since $\phi$ satisfies the linearized KP-I equation, $\eta$
satisfies
\[
\partial_{x}^{2}\left(  \frac{-D_{x}^{2}\eta\cdot\tau_{2}+D_{x}^{4}\eta
\cdot\tau_{2}-D_{y}^{2}\eta\cdot\tau_{2}}{\tau_{2}^{2}}\right)  =0.
\]
This together with $\left(  \ref{g2}\right)  $ implies
\[
-D_{x}^{2}\eta\cdot\tau_{2}+D_{x}^{4}\eta\cdot\tau_{2}-D_{y}^{2}\eta\cdot
\tau_{2}=0.
\]
The proof is completed.
\end{proof}

\section{Nongeneracy of the lump solution\label{lump}}

In this section, we  prove the nondegeneracy of the lump solution.

\subsection{Linearized B\"{a}cklund transformation between $\tau_{0}$ and
$\tau_{1}\label{Sec3}$}

In terms of $\tau_{0}$ and $\tau_{1},$ the B\"acklund transformation $\left(
\ref{b1}\right)  $ can be written as
\begin{equation}
\left\{
\begin{array}
[c]{l}%
\left(  D_{x}^{2}+\frac{1}{\sqrt{3}}D_{x}+\frac{1}{\sqrt{3}}iD_{y}\right)
\tau_{0}\cdot\tau_{1}=0,\\
\left(  -D_{x}-iD_{y}+D_{x}^{3}-\sqrt{3}iD_{x}D_{y}\right)  \tau_{0}\cdot
\tau_{1}=0.
\end{array}
\right.  \label{l1}%
\end{equation}
Linearizing this system at $\left(  \tau_{0},\tau_{1}\right)  $, we obtain
\begin{equation}
\left\{
\begin{array}
[c]{c}%
L_{1}\phi=G_{1}\eta,\\
M_{1}\phi=N_{1}\eta.
\end{array}
\right.  \label{s1}%
\end{equation}
Here for notational simplicity, we have defined%
\begin{align*}
L_{1}\phi &  =\left(  D_{x}^{2}+\frac{1}{\sqrt{3}}D_{x}+\frac{1}{\sqrt{3}%
}iD_{y}\right)  \phi\cdot\tau_{1},\\
M_{1}\phi &  =\left(  -D_{x}-iD_{y}+D_{x}^{3}-\sqrt{3}iD_{x}D_{y}\right)
\phi\cdot\tau_{1},
\end{align*}
and
\begin{align*}
G_{1}\eta &  =-\left(  D_{x}^{2}+\frac{1}{\sqrt{3}}D_{x}+\frac{1}{\sqrt{3}%
}iD_{y}\right)  \tau_{0}\cdot\eta,\\
N_{1}\eta &  =-\left(  -D_{x}-iD_{y}+D_{x}^{3}-\sqrt{3}iD_{x}D_{y}\right)
\tau_{0}\cdot\eta.
\end{align*}

\begin{proposition}
\label{P1}Let $\eta$ be a solution of the linearized bilinear KP-I equation at
$\tau_{1}:$
\begin{equation}
-D_{x}^{2}\eta\cdot\tau_{1}+D_{x}^{4}\eta\cdot\tau_{1}-D_{y}^{2}\eta\cdot
\tau_{1}=0. \label{yita1}%
\end{equation}
Suppose $\eta$ satisfies $\left(  \ref{es1}\right)  .$ Then the system
$\left(  \ref{s1}\right)  $ has a solution $\phi$ satisfying
\[
\left\vert \phi\right\vert +\left\vert \partial_{x}\phi\right\vert +\left\vert
\partial_{y}\phi\right\vert \leq C\left(  1+r\right)  ^{\frac{5}{2}}.
\]
Moreover,
\begin{equation}
-D_{x}^{2}\phi\cdot\tau_{0}+D_{x}^{4}\phi\cdot\tau_{0}=D_{y}^{2}\phi\cdot
\tau_{0}. \label{constant}%
\end{equation}

\end{proposition}

The rest of this sub-section will be devoted to the proof of Proposition \ref{P1}.

First of all, from the first equation in $\left(  \ref{s1}\right)  $, we get
\begin{equation}
\partial_{y}\phi\tau_{1}=i\left[  \partial_{x}\phi\tau_{1}+\sqrt{3}\left(
\partial_{x}^{2}\phi\tau_{1}-2\partial_{x}\phi\right)  \right]  -\sqrt
{3}iG_{1}\eta. \label{dy}%
\end{equation}
Inserting $\left(  \ref{dy}\right)  $ into the right hand side of the second
equation of $\left(  \ref{s1}\right)  $, we get
\begin{equation}
4\partial_{x}^{3}\phi\tau_{1}+\left(  2\sqrt{3}\tau_{1}-12\right)
\partial_{x}^{2}\phi+\left(  -4\sqrt{3}+\frac{12}{\tau_{1}}\right)
\partial_{x}\phi=F_{1}. \label{s2}%
\end{equation}
Here the right hand side is defined by
\begin{align*}
F_{1}  &  =3\partial_{x}\left(  G_{1}\eta\right)  +\sqrt{3}G_{1}\eta+N_{1}%
\eta-\frac{6}{\tau_{1}}G_{1}\eta\\
&  =-2\partial_{x}^{3}\eta+2\sqrt{3}i\partial_{x}\partial_{y}\eta-\frac
{6}{\tau_{1}}G_{1}\eta.
\end{align*}
To solve the equation $\left(  \ref{s2}\right)  $, we shall analyze the
solutions of the homogeneous equation%
\begin{equation}
2\tau_{1}^{2}g^{\prime\prime}+\left(  \sqrt{3}\tau_{1}-6\right)  \tau
_{1}g^{\prime}+\left(  6-2\sqrt{3}\tau_{1}\right)  g=0. \label{g}%
\end{equation}

\begin{lemma}
\label{h1}The equation $\left(  \ref{g}\right)  $ has two linearly independent
solutions
\[
g_{1}=\tau_{1}-\frac{\sqrt{3}}{2}\tau_{1}^{2},\text{ and }g_{2}=\tau
_{1}e^{-\frac{\sqrt{3}}{2}\tau_{1}}.
\]

\end{lemma}

\begin{proof}
This can be verified directly. Indeed, the equation $\left(  \ref{g}\right)  $
can be exactly solved using software such as \textit{Maple}. We will
frequently use this software in the rest of paper.
\end{proof}

Let $W$ be the Wronskian of the two solutions $g_{1},g_{2}.$ That is,
\[
W:=g_{1}\partial_{x}g_{2}-g_{2}\partial_{x}g_{1}=\frac{3}{4}\tau_{1}%
^{3}e^{-\frac{\sqrt{3}}{2}\tau_{1}}.
\]
By the variation of parameter formula, the equation
\[
4\tau_{1}g^{\prime\prime}+\left(  2\sqrt{3}\tau_{1}-12\right)  g^{\prime
}+\left(  \frac{12}{\tau_{1}}-4\sqrt{3}\right)  g=F_{1}%
\]
has a solution of the form
\[
g^{\ast}\left(  x,y\right)  =g_{2}\left(  x,y\right)  \int_{-\infty}^{x}%
\frac{g_{1}F_{1}}{4\tau_{1}W}ds-g_{1}\left(  x,y\right)  \int_{-\infty}%
^{x}\frac{g_{2}F_{1}}{4\tau_{1}W}ds.
\]
It follows that for each fixed $y,$ the equation
\[
4\tau_{1}\phi^{\prime\prime\prime}+\left(  2\sqrt{3}\tau_{1}-12\right)
\phi^{\prime\prime}+\left(  \frac{12}{\tau_{1}}-4\sqrt{3}\right)  \phi
^{\prime}=F
\]
has a solution of the form
\begin{equation}
w_{0}\left(  x,y\right)  =\int_{0}^{x}g^{\ast}\left(  s,y\right)  ds.
\label{fi}%
\end{equation}
We emphasize that $\frac{1}{\tau_{1}}$ has a singularity at the point $\left(
x,y\right)  =\left(  -\sqrt{3},0\right)  .$ Therefore we need to be very
careful about the behavior of $w_{0}$ around this singular point.

\begin{lemma}
Suppose $\eta$ satisfies $\left(  \ref{es1}\right)  .$ Then
\[
\left\vert w_{0}\right\vert \leq C\left(  1+r\right)  ^{\frac{5}{2}},\text{
for }x\leq10.
\]

\end{lemma}

\begin{proof}
Since $\eta$ satisfies $\left(  \ref{es1}\right)  ,$ we have%
\[
\left\vert F_{1}\right\vert \leq\left(  1+r\right)  ^{\frac{3}{2}}.
\]
Hence using the asymptotic behavior of $W$ and $g_{1},g_{2},$ we find that for
$r$ large,
\begin{align*}
\left\vert \frac{g_{1}F_{1}}{W\tau_{1}}\right\vert  &  \leq Ce^{\frac{\sqrt
{3}}{2}x}\left(  1+r\right)  ^{-\frac{1}{2}},\\
\left\vert \frac{g_{2}F_{1}}{W\tau_{1}}\right\vert  &  \leq C\left(
1+r\right)  ^{-\frac{3}{2}}.
\end{align*}
Near the singular point $\left(  -\sqrt{3},0\right)  $, using the fact that
\[
g_{2}-g_{1}=O\left(  \tau_{1}^{3}\right)  ,
\]
we infer
\[
\left\vert g^{\ast}\right\vert \leq C.
\]
As a consequence, for $x\leq10,$ we obtain
\[
\left\vert g^{\ast}\right\vert \leq C\left(  1+r\right)  ^{\frac{3}{2}}.
\]
It follows that for $x\leq10,$%
\begin{equation}
\left\vert w_{0}\right\vert \leq C\left(  1+r\right)  ^{\frac{5}{2}}.
\label{gr1}%
\end{equation}
We remark that as $x\rightarrow+\infty,$ since $g_{1}=O\left(  \tau_{1}%
^{2}\right)  ,$ $w_{0}$ may not satisfy $\left(  \ref{gr1}\right)  .$
\end{proof}

\begin{lemma}
\label{ki}The functions $\xi_{0}\left(  x,y\right)  :=1,$%
\[
\xi_{1}:=\frac{1}{2}\tau_{1}^{2}-\frac{\sqrt{3}}{6}\tau_{1}^{3},
\]
and
\[
\xi_{2}:=\left(  \frac{\sqrt{3}}{2}\tau_{1}+1\right)  e^{-\frac{\sqrt{3}}%
{2}x+\frac{\sqrt{3}}{4}yi}%
\]
solve the homogeneous system
\[
\left\{
\begin{array}
[c]{c}%
L_{1}\phi=0,\\
M_{1}\phi=0.
\end{array}
\right.
\]

\end{lemma}

\begin{proof}
This can be checked directly. Alternatively, we can look for solutions of
$L_{1}\phi=0$ in the form $f_{1}\left(  y\right)  \partial_{x}^{-1}g_{1}$ and
$f_{2}\left(  y\right)  \partial_{x}^{-1}g_{2}.$ This reduces to an ODE for
the unknown functions $f_{1}$ and $f_{2}$ and and solved.
\end{proof}

For given function $\eta,$ we have seen from $\left(  \ref{fi}\right)  $ that
the ODE $\left(  \ref{s2}\right)  $ for the unknown function $\phi$ can be
solved for each fixed $y.$ To solve the whole system $\left(  \ref{s1}\right)
,$ we define
\[
\Phi_{0}\left(  x,y\right)  :=L_{1}\phi-G_{1}\eta,
\]
and
\[
\Phi_{1}=\partial_{x}\Phi_{0},\Phi_{2}=\partial_{x}^{2}\Phi_{0}.
\]
Note that $\Phi_{i}$ actually depends on the function $\phi.$

Consider the system of equations
\begin{equation}
\left\{
\begin{array}
[c]{c}%
\Phi_{0}\left(  x,y\right)  =0,\\
\Phi_{1}\left(  x,y\right)  =0,\\
\Phi_{2}\left(  x,y\right)  =0,
\end{array}
\right.  \ \text{for }x=1. \label{1}%
\end{equation}
We seek a solution $\phi$ of $\left(  \ref{1}\right)  $ in the form
$w_{0}+w_{1},$ where $w_{0}$ is defined in $\left(  \ref{fi}\right)  $ and
\[
w_{1}\left(  x,y\right)  =\rho_{0}\left(  y\right)  \xi_{0}\left(  x,y\right)
+\rho_{1}\left(  y\right)  \xi_{1}\left(  x,y\right)  +\rho_{2}\left(
y\right)  \xi_{2}\left(  x,y\right)  ,
\]
for some unknown functions $\rho_{0},\rho_{1},\rho_{2}.$

\begin{lemma}
\label{initial}System $\left(  \ref{1}\right)  $ has a solution $\left(
\rho_{0},\rho_{1},\rho_{2}\right)  $ with the initial condition
\[
\rho_{i}\left(  0\right)  =0,\text{ }i=0,1,2.
\]

\end{lemma}

\begin{proof}
The equation $\Phi_{0}=0$ can be written as%
\[
L_{1}w_{1}=-L_{1}w_{0}+G_{1}\eta:=H_{0}.
\]
Similarly, we write $\Phi_{1}=0$ as%
\[
\partial_{x}\left[  L_{1}w_{1}\right]  =-\partial_{x}\left[  L_{1}%
w_{0}\right]  +\partial_{x}\left[  G_{1}\eta\right]  :=H_{1}.
\]
The equation $\Phi_{2}=0$ can be written as%
\[
\partial_{x}^{2}\left[  L_{1}w_{1}\right]  =-\partial_{x}^{2}\left[
L_{1}w_{0}\right]  +\partial_{x}^{2}\left[  G_{1}\eta\right]  :=H_{2}.
\]
Consider the system
\begin{equation}
\left\{
\begin{array}
[c]{l}%
L_{1}w_{1}=0,\\
\partial_{x}\left[  L_{1}w_{1}\right]  =0,\\
\partial_{x}^{2}\left[  L_{1}w_{1}\right]  =0.
\end{array}
\right.  \label{sys3}%
\end{equation}
In view of the definition of $w_{1},$ we know that $\left(  \ref{sys3}\right)
$ is a homogeneous system of first order differential equations for the
functions $\rho_{0},\rho_{1},\rho_{2}.$ Explicitly, $\left(  \ref{sys3}%
\right)  $ has the form
\[
\mathcal{A}\left(
\begin{array}
[c]{c}%
\rho_{0}^{\prime}\\
\rho_{1}^{\prime}\\
\rho_{2}^{\prime}%
\end{array}
\right)  =0,
\]
where
\[
\mathcal{A}=\left(
\begin{array}
[c]{ccc}%
\frac{i\tau_{1}\xi_{0}}{\sqrt{3}} & \frac{i\tau_{1}\xi_{1}}{\sqrt{3}} &
\frac{i\tau_{1}\xi_{2}}{\sqrt{3}}\\
\frac{i\tau_{1}\partial_{x}\xi_{0}}{\sqrt{3}} & \frac{i\tau_{1}\partial_{x}%
\xi_{1}}{\sqrt{3}} & \frac{i\tau_{1}\partial_{x}\xi_{2}}{\sqrt{3}}\\
\frac{i\tau_{1}\partial_{x}^{2}\xi_{0}}{\sqrt{3}} & \frac{i\tau_{1}%
\partial_{x}^{2}\xi_{1}}{\sqrt{3}} & \frac{i\tau_{1}\partial_{x}^{2}\xi_{2}%
}{\sqrt{3}}%
\end{array}
\right)  .
\]
Hence $\left(  \ref{1}\right)  $ has a solution
\[
\int_{0}^{y}\mathcal{A}^{-1}\left(  1,s\right)  \left(
\begin{array}
[c]{c}%
H_{0}\left(  1,s\right) \\
H_{1}\left(  1,s\right) \\
H_{2}\left(  1,s\right)
\end{array}
\right)  ds.
\]
$\allowbreak$This completes the proof.
\end{proof}

\begin{proposition}
\label{p1}Suppose $\eta$ satisfies $\left(  \ref{yita1}\right)  .$ Then%
\begin{align*}
\partial_{x}^{3}\Phi_{0}  &  =\left(  -\frac{\sqrt{3}}{2}+\frac{6}{\tau_{1}%
}\right)  \partial_{x}^{2}\Phi_{0}+\frac{1}{\tau_{1}}\left(  2\sqrt{3}%
-\frac{15}{\tau_{1}}\right)  \partial_{x}\Phi_{0}\\
&  +\frac{1}{\tau_{1}^{2}}\left(  \frac{15}{\tau_{1}}-2\sqrt{3}\right)
\Phi_{0}.
\end{align*}

\end{proposition}

\begin{proof}
One can verify this identity by direct computation. Although this tedious
calculation can be done by hand, we suggest to do it using computer softwares
such as \textit{Maple} or \textit{Mathematica}.

Intuitively, we expect that this identity follows from the compatibility
properties of suitable Lax pair of the KP-I equation. But up to now we have
not been able to rigorously show this.
\end{proof}

\begin{lemma}
\label{w0w1}The function $\phi=w_{0}+w_{1}$ satisfies $\left(  \ref{s1}%
\right)  $ for all $\left(  x,y\right)  \in\mathbb{R}^{2}.$ As a consequence,
it satisfies the linearized bilinear KP-I equation at $\tau_{0}:$%
\begin{equation}
-D_{x}^{2}\phi\cdot\tau_{0}+D_{x}^{4}\phi\cdot\tau_{0}-D_{y}^{2}\phi\cdot
\tau_{0}=0. \label{linear}%
\end{equation}

\end{lemma}

\begin{proof}
By the definition of $\Phi_{0},\Phi_{1},\Phi_{2}$ and Proposition \ref{p1}, we
have%
\[
\partial_{x}\left(
\begin{array}
[c]{c}%
\Phi_{0}\\
\Phi_{1}\\
\Phi_{2}%
\end{array}
\right)  =\left(
\begin{array}
[c]{ccc}%
0 & 1 & 0\\
0 & 0 & 1\\
\frac{\frac{15}{\tau_{1}}-2\sqrt{3}}{\tau_{1}^{2}} & \frac{2\sqrt{3}-\frac
{15}{\tau_{1}}}{\tau_{1}} & \frac{6}{\tau_{1}}-\frac{\sqrt{3}}{2}%
\end{array}
\right)  \left(
\begin{array}
[c]{c}%
\Phi_{0}\\
\Phi_{1}\\
\Phi_{2}%
\end{array}
\right)  .
\]
For each fixed $y\neq0,$ by Lemma \ref{initial}, $\Phi_{0}\left(  1,y\right)
=\Phi_{1}\left(  1,y\right)  =\Phi_{2}\left(  1,y\right)  =0.$ By the
uniqueness of solution to ODE, we obtain $\Phi_{0}\left(  x,y\right)
=\Phi_{1}\left(  x,y\right)  =\Phi_{2}\left(  x,y\right)  =0,$ for all
$x\in\mathbb{R}.$ On the other hand, by the definition of $w_{0}$ and $w_{1},$
the second equation of $\left(  \ref{s1}\right)  $ is also satisfied for all
$\left(  x,y\right)  \in\mathbb{R}^{2}$ with $y\neq0.$ By continuity of
$\phi,$ we know that $\phi=w_{0}+w_{1}$ satisfies $\left(  \ref{s1}\right)  $
for all $\left(  x,y\right)  \in\mathbb{R}^{2}.$ The equation $\left(
\ref{linear}\right)  $ then follows from the linearization of the bilinear
identity $\left(  \ref{back}\right)  .$
\end{proof}

\begin{lemma}
\label{ro}Suppose $\eta$ satisfies $\left(  \ref{es1}\right)  $. Let $\rho
_{i},i=0,1,2,$ be functions given by Lemma \ref{initial}. Then%
\[
\rho_{1}\left(  y\right)  =\rho_{2}\left(  y\right)  =0,\text{ for all }%
y\in\mathbb{R}.
\]

\end{lemma}

\begin{proof}
Dividing the equation $\Phi_{0}=0$ by $\xi_{2},$ we get
\begin{equation}
\frac{i}{\sqrt{3}\xi_{2}}\left(  \rho_{0}^{\prime}\xi_{0}+\rho_{1}^{\prime}%
\xi_{1}+\rho_{2}^{\prime}\xi_{2}\right)  \tau_{1}=\frac{1}{\xi_{2}}H_{0}.
\label{li}%
\end{equation}
For each fixed $y\in\mathbb{R},$ sending $x\rightarrow-\infty$ in $\left(
\ref{li}\right)  $ and using the estimate $\left(  \ref{es1}\right)  ,$ we
infer%
\[
\rho_{2}^{\prime}=0.
\]
This together with the initial condition $\rho_{2}\left(  0\right)  =0$ tell
us that $\rho_{2}=0.$

Now $\Phi_{0}=0$ becomes
\[
\frac{i}{\sqrt{3}}\left(  \rho_{0}^{\prime}\xi_{0}+\rho_{1}^{\prime}\xi
_{1}\right)  \tau_{1}=H_{0}.
\]
Dividing both sides of this equation by $\xi_{1}$ and letting $x\rightarrow
-\infty,$ we get $\rho_{1}^{\prime}=0.$ Hence $\rho_{1}=0.$ The proof is completed.
\end{proof}

\begin{proof}
[Proof of Proposition \ref{P1}]We have proved that $\rho_{1}$ and $\rho_{2}$
are both zero. Now let us define
\[
k\left(  y\right)  =\int_{-\infty}^{+\infty}\frac{g_{2}F_{1}}{4\tau_{1}%
W}ds,\text{ for }y\neq0.
\]
For each fixed $y,$ note that due to the definition of $w_{0},$ as
$x\rightarrow+\infty,$ the main order term of $w_{0}$ is $-k\left(  y\right)
\xi_{1}\left(  x,y\right)  .$ Write
\[
w_{0}+k\left(  y\right)  \xi_{1}\left(  x,y\right)  :=w^{\ast}.
\]
Then
\begin{equation}
\left\vert w^{\ast}\right\vert \leq C\left(  1+r\right)  ^{\frac{5}{2}},\text{
for }x\text{ large.} \label{star}%
\end{equation}

Now using similar arguments as that of Lemma \ref{ro} and the estimate
$\left(  \ref{star}\right)  ,$ we see that $k^{\prime}\left(  y\right)  =0$.
But $k\left(  y\right)  \rightarrow0$ as $\left\vert y\right\vert
\rightarrow+\infty.$ Hence $k\left(  y\right)  =0$. This together with
(\ref{star}) implies that
\[
\left\vert w_{0}\right\vert \leq C\left(  1+r\right)  ^{\frac{5}{2}}.
\]
Since $\rho_{1},\rho_{2}$ are both identically zero, we then infer from Lemma
\ref{w0w1} that the function $\phi=w_{0}+\rho_{0}\xi_{0}$ is the desired
solution. Note that the equation $\left(  \ref{constant}\right)  $ follows
from the linearization of the bilinear identity $\left(  \ref{back}\right)  .$
\end{proof}

\subsection{ Linearized B\"{a}cklund transformation between $\tau_{1}$ and
$\tau_{2}\label{Sec4}$}

In terms of $\tau_{1}$ and $\tau_{2},$ the B\"acklund transformation $\left(
\ref{b2}\right)  $ can be written as
\[
\left\{
\begin{array}
[c]{l}%
\left(  D_{x}^{2}-\frac{1}{\sqrt{3}}D_{x}+\frac{1}{\sqrt{3}}iD_{y}\right)
\tau_{1}\cdot\tau_{2}=0,\\
\left(  -D_{x}+iD_{y}+D_{x}^{3}-\sqrt{3}iD_{x}D_{y}\right)  \tau_{1}\cdot
\tau_{2}=0.
\end{array}
\right.
\]
The linearization of this system is%
\begin{equation}
\left\{
\begin{array}
[c]{c}%
L_{2}\phi=G_{2}\eta,\\
M_{2}\phi=N_{2}\eta.
\end{array}
\right.  \label{s3}%
\end{equation}
Here
\begin{align*}
L_{2}\phi &  =\left(  D_{x}^{2}-\frac{1}{\sqrt{3}}D_{x}+\frac{1}{\sqrt{3}%
}iD_{y}\right)  \phi\cdot\tau_{2},\\
M_{2}\phi &  =\left(  -D_{x}+iD_{y}+D_{x}^{3}-\sqrt{3}iD_{x}D_{y}\right)
\phi\cdot\tau_{2},
\end{align*}
and
\begin{align*}
G_{2}\eta &  =-\left(  D_{x}^{2}-\frac{1}{\sqrt{3}}D_{x}+\frac{1}{\sqrt{3}%
}iD_{y}\right)  \tau_{1}\cdot\eta,\\
N_{2}\eta &  =-\left(  -D_{x}+iD_{y}+D_{x}^{3}-\sqrt{3}iD_{x}D_{y}\right)
\tau_{1}\cdot\eta.
\end{align*}

\begin{proposition}
\label{P}Let $\eta=\eta\left(  x,y\right)  $ be a function solving the
linearized bilinear KP-I equation at $\tau_{2}:$%
\begin{equation}
-D_{x}^{2}\eta\cdot\tau_{2}+D_{x}^{4}\eta\cdot\tau_{2}=D_{y}^{2}\eta\cdot
\tau_{2}. \label{yita2}%
\end{equation}
Suppose $\eta$ satisfies $\left(  \ref{es1}\right)  .$ Then the system
$\left(  \ref{s1}\right)  $ has a solution $\phi$ with
\[
\left\vert \phi\right\vert +\left(  \left\vert \partial_{x}\phi\right\vert
+\left\vert \partial_{x}^{2}\phi\right\vert +\left\vert \partial_{x}%
\partial_{y}\phi\right\vert \right)  \left(  1+r\right)  \leq C\left(
1+r\right)  ^{\frac{5}{2}}.
\]
Moreover, $\phi$ satisfies
\[
-D_{x}^{2}\phi\cdot\tau_{1}+D_{x}^{4}\phi\cdot\tau_{1}=D_{y}^{2}\phi\cdot
\tau_{1}.
\]

\end{proposition}

From the first equation in $\left(  \ref{s3}\right)  $, we get
\begin{equation}
\partial_{y}\phi\tau_{2}=i\left[  \sqrt{3}\left(  \partial_{x}^{2}\phi\tau
_{2}-2\partial_{x}\phi\partial_{x}\tau_{2}\right)  -\partial_{x}\phi\tau
_{2}\right]  +2i\phi\bar{\tau}_{1}-\sqrt{3}iG\eta. \label{dfy}%
\end{equation}
Here we have used $\bar{\tau}_{1}$ to denote the complex conjugate of
$\tau_{1}.$ Inserting $\left(  \ref{dfy}\right)  $ into the second equation of
$\left(  \ref{s3}\right)  $, we obtain%

\begin{equation}
\partial_{x}^{3}\phi\tau_{2}+\left(  -\frac{\sqrt{3}}{2}\tau_{2}-6x\right)
\partial_{x}^{2}\phi+\left(  2\sqrt{3}x+\frac{12x^{2}}{\tau_{2}}\right)
\partial_{x}\phi-\frac{2\sqrt{3}x}{\tau_{2}}\bar{\tau}_{1}\phi=\frac{F_{2}}%
{4}. \label{Eq1}%
\end{equation}
Here
\[
F_{2}=N_{2}\eta-\left(  \frac{6\partial_{x}\tau_{2}}{\tau_{2}}G_{2}%
\eta-3\partial_{x}\left(  G_{2}\eta\right)  +\sqrt{3}G_{2}\eta\right)  .
\]

To solve equation $\left(  \ref{Eq1}\right)  ,$ we set%
\[
\phi=\tau_{1}\kappa\text{ and }h=\kappa^{\prime}.
\]
Equation $\left(  \ref{Eq1}\right)  $ is transformed into the equation
\begin{align*}
T\left(  h\right)   &  :=\tau_{1}\tau_{2}h^{\prime\prime}+\left[  3\tau
_{2}+\left(  -\frac{\sqrt{3}}{2}\tau_{2}-6x\right)  \tau_{1}\right]
h^{\prime}\\
&  +\left[  2\left(  -\frac{\sqrt{3}}{2}\tau_{2}-6x\right)  +\tau_{1}\left(
2\sqrt{3}x+\frac{12x^{2}}{\tau_{2}}\right)  \right]  h\\
&  =\frac{F_{2}}{4}.
\end{align*}

\begin{lemma}
The homogenous equation
\[
T\left(  h\right)  =0
\]
has two solutions $h_{1},h_{2},$ given by \qquad%
\begin{align*}
h_{1}\left(  x,y\right)   &  =\left(  x-yi\right)  ^{2}+\frac{2}{\sqrt{3}%
}\left(  x-yi\right)  +3+\frac{12\left(  y-\sqrt{3}i\right)  \left(
y+\frac{i}{\sqrt{3}}\right)  }{\tau_{1}^{2}}\\
&  =\frac{\tau_{2}}{3\tau_{1}^{2}}\left(  3\tau_{2}+4\sqrt{3}\left(
x+\bar{\tau}_{1}\right)  \right)  ,
\end{align*}
and
\[
h_{2}\left(  x,y\right)  =\frac{\tau_{2}}{\tau_{1}^{2}}\left(  x+yi-\sqrt
{3}\right)  e^{\frac{\sqrt{3}}{2}x}.
\]

\end{lemma}

\begin{proof}
This equation can be solved using \textit{Maple.}
\end{proof}

\begin{lemma}
\label{Tau2}The system%
\[
\left\{
\begin{array}
[c]{l}%
\left(  D_{x}^{2}-\frac{1}{\sqrt{3}}D_{x}+\frac{1}{\sqrt{3}}iD_{y}\right)
\phi\cdot\tau_{2}=0,\\
\left(  -D_{x}+iD_{y}+D_{x}^{3}-\sqrt{3}iD_{x}D_{y}\right)  \phi\cdot\tau
_{2}=0
\end{array}
\right.
\]
has three solutions $\zeta_{0},\zeta_{1},\zeta_{2},$ given by $\zeta_{0}%
=\tau_{1},$%
\begin{align*}
\zeta_{1}  &  =\tau_{1}\partial_{x}^{-1}h_{1}\\
&  =\tau_{1}\left(  \frac{\bar{z}^{3}}{3}+\frac{\bar{z}^{2}}{\sqrt{3}%
}+3x-11yi\right)  -12\left(  y-\sqrt{3}i\right)  \left(  y+\frac{i}{\sqrt{3}%
}\right)  ,
\end{align*}
and
\[
\zeta_{2}=e^{\sqrt{3}yi}\tau_{1}\partial_{x}^{-1}g_{2}=\left(  x^{2}-\frac
{8}{\sqrt{3}}x+y^{2}+\frac{4yi}{\sqrt{3}}+7\ \right)  e^{\frac{\sqrt{3}}%
{2}x+\sqrt{3}yi}.
\]

\end{lemma}

\begin{proof}
This is similar to Lemma \ref{ki} and can be checked by direct computation.
\end{proof}

Let $\tilde{W}$ be the Wronskian of $h_{1},h_{2}$. That is
\[
\tilde{W}=h_{1}\partial_{x}h_{2}-h_{2}\partial_{x}h_{1}=e^{\frac{\sqrt{3}}%
{2}x}\frac{\tau_{2}^{3}}{\tau_{1}^{3}}.
\]
Variation of parameter formula gives us a solution $\tau_{1}\partial_{x}%
^{-1}h^{\ast}$ to the equation $\left(  \ref{Eq1}\right)  ,$ where
\[
h^{\ast}=h_{2}\int_{+\infty}^{x}\frac{h_{1}}{\tilde{W}}\frac{F_{2}}{4\tau
_{1}\tau_{2}}ds-h_{1}\int_{+\infty}^{x}\frac{h_{2}}{\tilde{W}}\frac{F_{2}%
}{4\tau_{1}\tau_{2}}ds.
\]

\begin{lemma}
\bigskip Suppose $\eta$ satisfies $\left(  \ref{es1}\right)  .$ Let $\tilde
{w}_{0}=\tau_{1}\partial_{x}^{-1}h^{\ast}.$ Then
\[
\left\vert \tilde{w}_{0}\right\vert \leq C\left(  1+r\right)  ^{\frac{5}{2}%
},\text{ for }x\geq-10\text{.}%
\]

\end{lemma}

\begin{proof}
Since $\eta$ satisfies $\left(  \ref{es1}\right)  ,$ we have%
\[
\left\vert F_{2}\right\vert \leq C\left(  1+r\right)  ^{\frac{5}{2}}.
\]
This together with the explicit formula of $\tau_{1},\tau_{2}$ imply that
\begin{align*}
\left\vert \frac{h_{1}F_{2}}{\tilde{W}\tau_{1}\tau_{2}}\right\vert  &  \leq
Ce^{-\frac{\sqrt{3}x}{2}}\left(  1+r\right)  ^{-\frac{3}{2}},\\
\left\vert \frac{h_{2}F_{2}}{\tilde{W}\tau_{1}\tau_{2}}\right\vert  &  \leq
C\left(  1+r\right)  ^{-\frac{5}{2}}.
\end{align*}
We then deduce
\[
\left\vert h^{\ast}\left(  x,y\right)  \right\vert \leq\left(  1+r\right)
^{\frac{1}{2}},\text{ for }x\geq-10.
\]
From this upper bound, it follows that
\[
\left\vert \tilde{w}_{0}\right\vert \leq C\left(  1+r\right)  ^{\frac{5}{2}%
},\text{ for }x\geq-10.
\]

\end{proof}

Let $\eta$ be a function satisfying $\left(  \ref{yita2}\right)  .$ Slightly
abusing the notation, we define
\[
\Phi_{0}=L_{2}\phi-G_{2}\eta.
\]
Also define $\Phi_{1}=\partial_{x}\Phi_{0}$ and $\Phi_{2}=\partial_{x}^{2}%
\Phi_{0}.$

We would like find a solution $\phi$ for the system
\begin{equation}
\left\{
\begin{array}
[c]{c}%
\Phi_{0}=0\\
\Phi_{1}=0\\
\Phi_{2}=0
\end{array}
,\text{ for }x=0.\right.  \label{P2}%
\end{equation}
Similarly as before, we seek a solution of this problem with the form
$\tilde{w}_{0}+\tilde{w}_{1},$ with
\[
\tilde{w}_{1}=\beta_{0}\left(  y\right)  \zeta_{0}+\beta_{1}\left(  y\right)
\zeta_{1}+\beta_{2}\left(  y\right)  \zeta_{2},
\]
where $\beta_{0},\beta_{1},\beta_{2}$ are functions of $y$ to be determined.

The problem $\left(  \ref{P2}\right)  $ can be written as
\begin{equation}
\left\{
\begin{array}
[c]{l}%
L_{2}\tilde{w}_{1}=H_{0},\\
\partial_{x}\left(  L_{2}\tilde{w}_{1}\right)  =H_{1},\\
\partial_{x}^{2}\left(  L_{2}\tilde{w}_{1}\right)  =H_{2}.
\end{array}
\right.  \label{P3}%
\end{equation}
Here
\begin{align*}
H_{0}  &  =G_{2}\eta-L_{2}\tilde{w}_{0},\\
H_{1}  &  =\partial_{x}\left(  G_{2}\eta-L_{2}\tilde{w}_{0}\right)  ,\\
H_{2}  &  =\partial_{x}^{2}\left(  G_{2}\eta-L_{2}\tilde{w}_{0}\right)  .
\end{align*}

\begin{lemma}
\label{beta}The equation $\left(  \ref{P3}\right)  $ has a solution $\left(
\beta_{0},\beta_{1},\beta_{2}\right)  $ satisfying the initial condition
\[
\beta_{i}\left(  0\right)  =0,i=0,1,2.
\]

\end{lemma}

\begin{proof}
The proof is similar to that of Lemma \ref{initial}. We omit the details.
\end{proof}

\begin{proposition}
\label{P4}Suppose $\eta$ satisfies $\left(  \ref{yita2}\right)  .$ Then%
\begin{align*}
\partial_{x}^{3}\Phi_{0}  &  =\left(  \frac{12x}{\tau_{2}}+\frac{\sqrt{3}}%
{2}\right)  \partial_{x}^{2}\Phi_{0}+\left(  \frac{6}{\tau_{2}}-\frac
{4\sqrt{3}x}{\tau_{2}}-\frac{60x^{2}}{\tau_{2}^{2}}\right)  \partial_{x}%
\Phi_{0}\\
&  +\left(  2\sqrt{3}\frac{x\bar{\tau}_{1}}{\tau_{2}^{2}}-\frac{\sqrt{3}}%
{\tau_{2}}-\frac{36x}{\tau_{2}^{2}}+\frac{8\sqrt{3}x^{2}}{\tau_{2}^{2}}%
+\frac{120x^{3}}{\tau_{2}^{3}}\right)  \Phi_{0}.
\end{align*}

\end{proposition}

\begin{proof}
This is similar to Propositon \ref{p1} and can be checked directly using Maple.
\end{proof}

\begin{lemma}
The function $\phi=\tilde{w}_{0}+\tilde{w}_{1}$ solves the system $\left(
\ref{s3}\right)  $ for all $\left(  x,y\right)  \in\mathbb{R}^{2}.$
\end{lemma}

\begin{proof}
By Proposition \ref{P4},
\[
\partial_{x}\left(
\begin{array}
[c]{c}%
\Phi_{0}\\
\Phi_{1}\\
\Phi_{2}%
\end{array}
\right)  =\left(
\begin{array}
[c]{ccc}%
0 & 1 & 0\\
0 & 0 & 1\\
a_{31} & a_{32} & a_{33}%
\end{array}
\right)  \left(
\begin{array}
[c]{c}%
\Phi_{0}\\
\Phi_{1}\\
\Phi_{2}%
\end{array}
\right)  ,
\]
where%
\begin{align*}
a_{31}  &  =2\sqrt{3}\frac{x\bar{\tau}_{1}}{\tau_{2}^{2}}-\frac{\sqrt{3}}%
{\tau_{2}}-\frac{36x}{\tau_{2}^{2}}+\frac{8\sqrt{3}x^{2}}{\tau_{2}^{2}}%
+\frac{120x^{3}}{\tau_{2}^{3}},\\
a_{32}  &  =\frac{6}{\tau_{2}}-\frac{4\sqrt{3}x}{\tau_{2}}-\frac{60x^{2}}%
{\tau_{2}^{2}},
\end{align*}
and
\[
a_{33}=\frac{12x}{\tau_{2}}+\frac{\sqrt{3}}{2}.
\]
For each fixed $y,$ since $\Phi_{i}\left(  0,y\right)  =0,i=0,1,2,$ we deduce
from the uniqueness of solutions to ODE that $\Phi_{i}\left(  x,y\right)  =0,$
for all $x\in\mathbb{R}.$ This finishes the proof.
\end{proof}

\begin{lemma}
\label{ro1}Let $\beta_{i},i=0,1,2$ be the functions given by Lemma \ref{beta}.
Then $\beta_{1}=\beta_{2}=0.$
\end{lemma}

\begin{proof}
The proof is similar to that of Lemma \ref{ro}, using the asymptotic behavior
of $\zeta_{1},\zeta_{2}$ as $x\rightarrow-\infty.$
\end{proof}

With Lemma \ref{ro1} at hand, we can prove Proposition \ref{P} similarly as
before, we omit the details.

\subsection{Proof of the nondegeneracy of the lump}

With the previous preparations, we proceed to the proof of the nondegeneracy
of the lump solution. In this section, we denote $x+yi$ by $z.$

\begin{lemma}
\label{L4}Suppose $\eta$ satisfies%
\[
\left\{
\begin{array}
[c]{c}%
L_{1}\phi=G_{1}\eta,\\
M_{1}\phi=N_{1}\eta,
\end{array}
\right.
\]
Then
\[
-4\partial_{x}^{3}\eta+2\sqrt{3}\partial_{x}^{2}\eta=\Theta_{1}\phi,
\]
where
\[
\Theta_{1}\phi:=-M_{1}\phi-\sqrt{3}L_{1}\phi+3\partial_{x}\left(  L_{1}%
\phi\right)
\]
In particular, if
\[
G_{1}\eta=N_{1}\eta=0,
\]
and
\begin{equation}
\left\vert \eta\right\vert \leq C\left(  1+r\right)  ^{\frac{5}{2}}.
\label{gr}%
\end{equation}
Then $\eta=c_{1}+c_{2}\tau_{1},$ for some constants $c_{1},c_{2}.$
\end{lemma}

\begin{proof}
The equation $G_{1}\eta=L_{1}\phi$ is
\[
\left(  D_{x}^{2}+\frac{1}{\sqrt{3}}D_{x}+\frac{i}{\sqrt{3}}D_{y}\right)
\tau_{0}\cdot\eta=-L_{1}\phi.
\]
That is,
\[
\partial_{x}^{2}\eta+\frac{1}{\sqrt{3}}\left(  -\partial_{x}\eta\right)
+\frac{i}{\sqrt{3}}\left(  -\partial_{y}\eta\right)  =-L_{1}\phi.
\]
Inserting this identity into the equation
\[
\left(  -D_{x}-iD_{y}+D_{x}^{3}-\sqrt{3}iD_{x}D_{y}\right)  \tau_{0}\cdot
\eta=-M_{1}\phi,
\]
we get
\[
\sqrt{3}\partial_{x}^{2}\eta-\partial_{x}^{3}\eta-\sqrt{3}i\partial_{x}\left(
-\sqrt{3}i\left(  \partial_{x}^{2}\eta-\frac{1}{\sqrt{3}}\partial_{x}%
\eta+L_{1}\phi\right)  \right)  =-M_{1}\phi-\sqrt{3}L_{1}\phi.
\]
Hence%
\[
-4\partial_{x}^{3}\eta+2\sqrt{3}\partial_{x}^{2}\eta=-M_{1}\phi-\sqrt{3}%
L_{1}\phi+3\partial_{x}\left(  L_{1}\phi\right)  .
\]
If $L_{1}\phi=M_{1}\phi=0,$ then
\[
-4\partial_{x}^{3}\eta+2\sqrt{3}\partial_{x}^{2}\eta=0.
\]
The solutions of this equation are given by%
\[
c_{1}+c_{2}x+c_{3}e^{\frac{\sqrt{3}}{2}x},
\]
where $c_{1},c_{2},c_{3}$ are constants which may depend on $y.$ Due to the
growth estimate of $\eta,$ we find that
\[
\eta=c_{1}+c_{2}x.
\]
Inserting this into the equation $G_{1}\eta=0,$ we find that $\eta$ is a
linear combination of $1$ and $\tau_{1}.$
\end{proof}

\begin{lemma}
We have%
\[
\Theta_{1}\left(  x\right)  =4\sqrt{3},\text{ }\Theta_{1}\left(  y\right)
=0.
\]
and%
\[
\Theta_{1}\left(  x^{2}-y^{2}\right)  =8\sqrt{3}x,\text{ }\Theta_{1}\left(
xy\right)  =2\sqrt{3}i\bar{\tau}.
\]

\end{lemma}

\begin{proof}
This follows from direct computation. For instances, we compute
\begin{align*}
L_{1}\left(  x^{2}-y^{2}\right)   &  =\left(  D_{x}^{2}+\frac{1}{\sqrt{3}%
}D_{x}+\frac{1}{\sqrt{3}}iD_{y}\right)  \left(  x^{2}-y^{2}\right)  \cdot
\tau_{1}\\
&  =2\tau_{1}-4x+\frac{1}{\sqrt{3}}\left(  2x\tau_{1}-\left(  x^{2}%
-y^{2}\right)  \right)  +\frac{i}{\sqrt{3}}\left(  -2y\tau_{1}-i\left(
x^{2}-y^{2}\right)  \right) \\
&  =\frac{2}{\sqrt{3}}\tau_{2}.
\end{align*}%
\begin{align*}
M_{1}\left(  x^{2}-y^{2}\right)   &  =\left(  -D_{x}-iD_{y}+D_{x}^{3}-\sqrt
{3}iD_{x}D_{y}\right)  \left(  x^{2}-y^{2}\right)  \cdot\tau_{1}\\
&  =-\left(  2x\tau_{1}-\left(  x^{2}-y^{2}\right)  \right)  -i\left(
-2y\tau_{1}-\left(  x^{2}-y^{2}\right)  i\right) \\
&  +\left(  -3\right)  2-\sqrt{3}i\left(  -2xi-\left(  -2y\right)  \right) \\
&  =-2\tau_{2}-4\sqrt{3}x.
\end{align*}
It follows immediately that
\begin{align*}
\Theta_{1}\left(  x^{2}-y^{2}\right)   &  =2\tau_{2}+4\sqrt{3}x-\sqrt{3}%
\frac{2}{\sqrt{3}}\tau_{2}+3\frac{4x}{\sqrt{3}}\\
&  =8\sqrt{3}x.
\end{align*}

\end{proof}

\begin{lemma}
Define $\digamma\left(  \phi\right)  :=\left(  L_{1}\left(  \phi\right)
,M_{1}\left(  \phi\right)  \right)  ,$ $\mathcal{J}\left(  \phi\right)
:=\left(  G_{1}\left(  \phi\right)  ,N_{1}\left(  \phi\right)  \right)  .$
Then
\[
\digamma\left(  x\right)  =\mathcal{J}\left(  x\tau_{1}-\sqrt{3}z\right)  ,
\]%
\[
\digamma\left(  y\right)  =\mathcal{J}\left(  y\tau_{1}\right)  ,
\]
and
\begin{equation}
\digamma\left(  x^{2}-y^{2}\right)  =\mathcal{J}\left(  \rho_{1}\right)  ,
\label{xy}%
\end{equation}%
\[
\digamma\left(  xy\right)  =\mathcal{J}\left(  \rho_{2}\right)  ,
\]
where
\begin{align*}
\rho_{1}  &  =\frac{2}{3}x^{3}+\frac{4}{3}\sqrt{3}x^{2}-\frac{4\sqrt{3}}%
{3}yix-\frac{2i}{3}y^{3}+\frac{2}{3}\sqrt{3}y^{2}-14yi,\\
\rho_{2}  &  =\frac{1}{2}x^{2}y+\frac{5}{6}i\sqrt{3}x^{2}+\frac{1}{6}%
ix^{3}\allowbreak+\left(  \frac{y^{2}}{2}i+\frac{\sqrt{3}}{3}y\right)
x+\frac{y^{3}}{6}+\frac{\sqrt{3}}{6}y^{2}i+5y.
\end{align*}

\end{lemma}

\begin{proof}
We only prove $\left(  \ref{xy}\right)  .$ The proof of other cases are
similar. Consider the equation
\[
-4\partial_{x}^{3}\eta+2\sqrt{3}\partial_{x}^{2}\eta=\Theta_{1}\left(
x^{2}-y^{2}\right)  =8\sqrt{3}x.
\]
This equation has a solution
\[
\rho_{1}=\frac{2}{3}x^{3}+\frac{4}{3}\sqrt{3}x^{2}+a\left(  y\right)
x+b\left(  y\right)  ,
\]
where $a\left(  y\right)  $ and $b\left(  y\right)  $ are functions to be
determined. Since
\[
\partial_{x}^{2}\rho_{1}+\frac{1}{\sqrt{3}}\left(  -\partial_{x}\rho
_{1}\right)  +\frac{i}{\sqrt{3}}\left(  -\partial_{y}\rho_{1}\right)
=-L_{1}\left(  x^{2}-y^{2}\right)  =-2\frac{x^{2}+y^{2}+3}{\sqrt{3}},
\]
we get%
\[
4x+\frac{8\sqrt{3}}{3}-\frac{1}{\sqrt{3}}\left(  2x^{2}+\frac{8}{3}\sqrt
{3}x+a\right)  -\frac{i}{\sqrt{3}}\left(  a^{\prime}x+b^{\prime}\right)
=-2\frac{x^{2}+y^{2}+3}{\sqrt{3}}.
\]
Hence%
\[
a\left(  y\right)  =-\frac{4\sqrt{3}}{3}yi,b\left(  y\right)  =-\frac{2i}%
{3}y^{3}+\frac{2}{3}\sqrt{3}y^{2}-14yi.
\]
From here we get $\rho_{1}$ immediately.
\end{proof}

\begin{lemma}
\label{L3}Suppose $\eta$ satisfies
\[
\left\{
\begin{array}
[c]{c}%
L_{2}\phi=G_{2}\eta,\\
M_{2}\phi=N_{2}\eta.
\end{array}
\right.
\]
Then
\[
\partial_{x}^{3}\eta\tau_{1}+\left(  \frac{\sqrt{3}}{2}\tau_{1}-3\right)
\partial_{x}^{2}\eta+\left(  \frac{3}{\tau_{1}}-\sqrt{3}\right)  \partial
_{x}\eta+\frac{\sqrt{3}}{\tau_{1}}\eta=\Theta_{2}\left(  \phi\right)  .
\]
Here
\[
\Theta_{2}\left(  \phi\right)  :=\frac{1}{4}\left(  \frac{6}{\tau_{1}}%
L_{2}\phi-3\partial_{x}\left(  L_{2}\phi\right)  +M_{2}\phi+\sqrt{3}L_{2}%
\phi\right)  .
\]

\end{lemma}

\begin{proof}
Explicitly, $\eta$ satisfies%
\begin{equation}
\left\{
\begin{array}
[c]{l}%
\left(  D_{x}^{2}-\frac{1}{\sqrt{3}}D_{x}+\frac{1}{\sqrt{3}}iD_{y}\right)
\tau_{1}\cdot\eta=-L_{2}\phi,\\
\left(  -D_{x}+iD_{y}+D_{x}^{3}-\sqrt{3}iD_{x}D_{y}\right)  \tau_{1}\cdot
\eta=-M_{2}\phi.
\end{array}
\right.  \label{y2}%
\end{equation}
Let us write the first equation in this system as%
\[
\partial_{y}\eta\tau_{1}=-i\left[  \partial_{x}\eta\tau_{1}+\sqrt{3}\left(
\partial_{x}^{2}\eta\tau_{1}-2\partial_{x}\eta\right)  -2\eta\right]
-\sqrt{3}iL_{2}\phi.
\]
Inserting this identity into the right hand side of the second equation the
system $\left(  \ref{y2}\right)  $, we get
\[
\partial_{x}^{3}\eta\tau_{1}+\left(  \frac{\sqrt{3}}{2}\tau_{1}-3\right)
\partial_{x}^{2}\eta+\left(  \frac{3}{\tau_{1}}-\sqrt{3}\right)  \partial
_{x}\eta+\frac{\sqrt{3}}{\tau_{1}}\eta=\Theta_{2}.
\]

\end{proof}

Note that $\eta=\tau_{2}$ satisfies the homogeneous equation%
\begin{equation}
\partial_{x}^{3}\eta\tau_{1}+\left(  \frac{\sqrt{3}}{2}\tau_{1}-3\right)
\partial_{x}^{2}\eta+\left(  \frac{3}{\tau_{1}}-\sqrt{3}\right)  \partial
_{x}\eta+\frac{\sqrt{3}}{\tau_{1}}\eta=0. \label{yita}%
\end{equation}
Letting $\eta=\tau_{2}\kappa$ and $p=\kappa^{\prime},$ equation $\left(
\ref{yita}\right)  $ becomes
\begin{equation}
\tau_{1}\tau_{2}p^{\prime\prime}+\left(  6x\tau_{1}+\left(  \frac{\sqrt{3}}%
{2}\tau_{1}-3\right)  \tau_{2}\right)  p^{\prime}+\left(  6\tau_{1}+2x\left(
\sqrt{3}\tau_{1}-6\right)  +\left(  \frac{3}{\tau_{1}}-\sqrt{3}\right)
\tau_{2}\right)  p=0. \label{g4}%
\end{equation}

\begin{lemma}
\label{homo}The equation $\left(  \ref{g4}\right)  $ has two solutions given
by
\[
p_{1}=\frac{\left(  x+yi\right)  ^{2}-3}{\tau_{2}^{2}}%
\]
and%
\[
p_{2}:=\left(  8\sqrt{3}x-4i\sqrt{3}y+3x^{2}+3y^{2}+21\right)  \allowbreak
e^{-\frac{1}{2}\sqrt{3}x}\frac{\tau_{1}}{\tau_{2}^{2}}.
\]
In particular, if $\eta$ satisfies $\left(  \ref{gr}\right)  $ and
\[
G_{2}\eta=N_{2}\eta=0,
\]
Then $\eta=c_{1}z+c_{2}\tau_{2}.$
\end{lemma}

Note that $\partial_{x}^{-1}p_{1}=-\frac{z}{\tau_{2}}.$ Hence we get a
solution $\eta=-z$ for the equation $\left(  \ref{yita}\right)  .$ This
solution is corresponding to the translation of $\tau_{2}$ along the $x$ and
$y$ axes.

\begin{lemma}
\label{th}We have
\[
M_{2}\left(  y\tau_{1}\right)  =-2\sqrt{3}xy+3x^{2}y+3xi+2\sqrt{3}y^{2}%
i-x^{3}i+3xy^{2}i+9y-y^{3}-6\sqrt{3}i,
\]%
\[
M_{2}\left(  z^{2}\right)  =6x^{2}yi-2\sqrt{3}x^{2}-2\sqrt{3}y^{2}%
+2x^{3}-6xy^{2}+12yi-2y^{3}i+6\sqrt{3},
\]%
\begin{align*}
L_{2}\rho_{1}  &  =24x+12\sqrt{3}xyi-\frac{4\sqrt{3}}{9}x^{3}yi+4xy^{2}%
+\frac{4\sqrt{3}}{3}x^{2}y^{2}-\frac{8}{3}y^{3}i+22\sqrt{3}+\frac{4\sqrt{3}%
}{3}x^{2}\\
&  +\frac{4\sqrt{3}}{9}xy^{3}i+\frac{4\sqrt{3}}{3}y^{2}-28yi+\frac{2\sqrt{3}%
}{9}x^{4}-\frac{4}{3}x^{3}+\frac{2\sqrt{3}}{9}y^{4}.
\end{align*}%
\begin{align*}
M_{2}\rho_{1}  &  =-42+\frac{4\sqrt{3}}{3}y^{3}i-\frac{4\sqrt{3}}{3}%
x^{3}+4\sqrt{3}x^{2}yi-\frac{4}{3}xy^{3}i+\frac{4}{3}yx^{3}i\\
&  -4xyi+4\sqrt{3}xy^{2}-8\sqrt{3}yi-\frac{2}{3}y^{4}-\frac{2}{3}x^{4}%
-4x^{2}y^{2}+16\sqrt{3}x+16y^{2}-24x^{2},
\end{align*}%
\begin{align*}
L_{2}\rho_{2}  &  =13y+10\sqrt{3}i-\frac{2\sqrt{3}}{9}x^{3}y-4\sqrt{3}%
xy-x^{2}y+xy^{2}i+\frac{5}{3}y^{3}+\frac{2\sqrt{3}}{9}x^{4}i\\
&  +\frac{2\sqrt{3}}{9}y^{4}i+\frac{2\sqrt{3}}{9}xy^{3}+\frac{4\sqrt{3}}%
{3}y^{2}i-\frac{1}{3}x^{3}i+9xi-\frac{2\sqrt{3}}{3}x^{2}i,
\end{align*}%
\begin{align*}
M_{2}\rho_{2}  &  =7y^{2}i-3\sqrt{3}x^{2}y-2xy+\frac{2}{3}x^{3}y+5\sqrt
{3}y-9x^{2}i+\sqrt{3}xy^{2}i\\
&  -15i-\frac{2}{3}y^{4}i-\frac{2}{3}x^{4}i-\frac{2}{3}xy^{3}+\sqrt{3}%
xi-\frac{\sqrt{3}}{3}ix^{3}-\frac{\sqrt{3}}{3}y^{3}.
\end{align*}

\end{lemma}

\begin{proof}
Direct computation.
\end{proof}

We are now in a position to prove Theorem \ref{main}.

\begin{proof}
[Proof of Theorem \ref{main}]Let $\phi$ be a solution of $\left(
\ref{FI}\right)  $ satisfying the assumption of Theorem \ref{main}. Then using
Lemma \ref{core}, we can find $\eta_{2}$, a solution of $\left(
\ref{yita2}\right)  ,$ satisfying $\left(  \ref{es1}\right)  .$ In view of
Lemma \ref{homo}, if $G_{2}\eta_{2}=N_{2}\eta_{2}=0.$ Then $\eta_{2}%
=c_{1}z+c_{2}\tau_{2}.$ Therefore, to prove the theorem, from now on, we can
assume $G_{2}\eta_{2}\neq0$ or $N_{2}\eta_{2}\neq0.$

By Proposition \ref{P}, there exists a solution $\eta_{1}$ of the equation
\[
\left(  D_{x}^{2}-D_{x}^{4}+D_{y}^{2}\right)  \eta_{1}\cdot\tau_{1}=0,
\]
satisfying the estimate $\left(  \ref{es1}\right)  .$

Case 1. $G_{1}\eta_{1}=N_{1}\eta_{1}=0.$

In this case, by Lemma \ref{L4}, $\eta_{1}=a_{1}+a_{2}\tau_{1}.$ Accordingly,
\[
\eta_{2}=d_{1}\partial_{x}\tau_{2}+d_{2}\partial_{y}\tau_{2}+d_{3}\tau_{2},
\]
for some constants $d_{1},d_{2},d_{3}.$

Case 2. $G_{1}\eta_{1}\neq0$ or $N_{1}\eta_{1}\neq0.$

In this case, by Proposition \ref{P1}, there exists a solution $\eta_{0}$ of
\begin{equation}
\left(  D_{x}^{2}-D_{x}^{4}+D_{y}^{2}\right)  \eta_{0}\cdot\tau_{0}=0,
\label{eq2}%
\end{equation}
satisfying
\begin{equation}
\left\vert \eta_{0}\right\vert +\left\vert \partial_{x}\eta_{0}\right\vert
+\left\vert \partial_{y}\eta_{0}\right\vert \leq C\left(  1+r\right)
^{\frac{5}{2}}. \label{eq1}%
\end{equation}
From $\left(  \ref{eq2}\right)  $ and $\left(  \ref{eq1}\right)  ,$ we infer
that
\[
\partial_{x}^{2}\eta_{0}+\partial_{y}^{2}\eta_{0}=0.
\]
Therefore, for some constants $c_{1},...,c_{5},$
\[
\eta_{0}=c_{1}+c_{2}x+c_{3}y+c_{4}\left(  x^{2}-y^{2}\right)  +c_{5}xy.
\]
We claim that $c_{2}=c_{3}=c_{4}=c_{5}=0.$ Indeed, if $c_{4}$ or $c_{5}$ is
nonzero, then using Lemma \ref{th}, we find that $L_{2}\left(  \eta
_{1}\right)  $ and $M_{2}\left(  \eta_{1}\right)  $ will grow like $x^{4}.$
This contradicts with the asymptotic behavior $\left(  \ref{es1}\right)  $ of
$\eta_{2}.$ On the other hand, if $c_{2}$ or $c_{3}$ is nonzero, then still by
Lemma \ref{th}, $M_{2}\left(  \eta_{1}\right)  $ will grow like $x^{3}$ or
$x^{2}y.$ This also contradicts with the asymptotic behavior of $\eta_{2}.$
Hence $\eta_{0}$ is a constant. Then from the previous discussion, we deduce
that
\[
\eta_{2}=d_{1}\partial_{x}\tau_{2}+d_{2}\partial_{y}\tau_{2}+d_{3}\tau_{2},
\]
for some constants $d_{1},d_{2},d_{3}.$ The proof is thus finished.
\end{proof}

\section{Nondegeneracy of a family of $y$-periodic solutions\label{period}}

Following similar arguments as in the previous sections, we would like to show
the nondegeneracy of a family of $y$-periodic solutions naturally associated
to the lump solution.

Let $k\in\left(  0,1\right)  $ and $p=\sqrt{1-k^{2}}i.$ In some cases, we also
denote $\sqrt{1-k^{2}}$ by $b$. Define
\[
\tilde{\iota}=e^{\frac{k}{2}\left(  x-py-t\right)  }+e^{-\frac{k}{2}\left(
x-py-t\right)  }.
\]
Note that actually $\tilde{\iota}$ is depending on $k$ and $p.$

\begin{lemma}
$\tilde{\iota}$ satisfies the bilinear KP-I equation.
\end{lemma}

\begin{proof}
This follows from direct computation. We omit the details.
\end{proof}

Remark that by choosing $k=1,$ we get the corresponding function
\[
2\partial_{x}^{2}\ln\tilde{\iota}=\frac{1}{2\cosh^{2}\frac{x-t}{2}}.
\]
This is a solution for the classical KdV equation.

\subsection{The linearized B\"{a}cklund transformation between $\iota_{0}$ and
$\iota$}

Next we consider the B\"{a}cklund transformation between $\iota_{0}=1$ and
$\tilde{\iota}.$ Since many ideas are similar as in the previous sections, in
certain places, we will omit some of the details.

Define the parameters
\[
\lambda=\frac{k^{2}}{4},\mu=\frac{pi}{\sqrt{3}},
\]
where as usual, the notation $i$ represents the imaginary unit.

\begin{lemma}
The B\"{a}cklund transformation between $\iota_{0}$ and $\tilde{\iota}$ is
given by
\begin{equation}
\left\{
\begin{array}
[c]{l}%
\left(  D_{x}^{2}+\mu D_{x}+\frac{1}{\sqrt{3}}iD_{y}\right)  \iota_{0}%
\cdot\tilde{\iota}-\lambda\iota_{0}\tilde{\iota}=0,\\
\left(  D_{t}+3\lambda D_{x}-\sqrt{3}\mu iD_{y}+D_{x}^{3}-\sqrt{3}iD_{x}%
D_{y}-\frac{3k^{2}\mu}{4}\right)  \iota_{0}\cdot\tilde{\iota}=0.
\end{array}
\right.  \label{B1}%
\end{equation}

\end{lemma}

\begin{proof}
Let $\eta_{1}=k\left(  x-py-t\right)  .$ We have
\begin{align*}
&  \left(  D_{x}^{2}+\mu D_{x}+\frac{1}{\sqrt{3}}iD_{y}\right)  \iota_{0}%
\cdot\tilde{\iota}-\lambda\iota_{0}\tilde{\iota}\\
&  =\partial_{x}^{2}\tilde{\iota}+\mu\left(  -\partial_{x}\tilde{\iota
}\right)  +\frac{1}{\sqrt{3}}i\left(  -\partial_{y}\tilde{\iota}\right)
-\lambda\tilde{\iota}\\
&  =e^{-\frac{\eta_{1}}{2}}\left(  \frac{k^{2}}{4}+\frac{\mu k}{2}-\frac
{pki}{2\sqrt{3}}-\lambda\right) \\
&  +e^{\frac{\eta_{1}}{2}}\left(  \frac{k^{2}}{4}-\frac{\mu k}{2}+\frac
{pki}{2\sqrt{3}}-\lambda\right)  .
\end{align*}
Since
\[
\lambda=\frac{k^{2}}{4},\mu=\frac{pi}{\sqrt{3}},
\]
we get
\begin{align*}
\frac{k^{2}}{4}+\frac{\mu k}{2}-\frac{pki}{2\sqrt{3}}-\lambda &  =0,\\
\frac{k^{2}}{4}-\frac{\mu k}{2}+\frac{pki}{2\sqrt{3}}-\lambda &  =0.
\end{align*}
Hence
\[
\left(  D_{x}^{2}+\mu D_{x}+\frac{1}{\sqrt{3}}iD_{y}\right)  \iota_{0}%
\cdot\tilde{\iota}-\lambda\iota_{0}\tilde{\iota}=0.
\]
The second equation
\[
\left(  D_{t}+3\lambda D_{x}-\sqrt{3}\mu iD_{y}+D_{x}^{3}-\sqrt{3}iD_{x}%
D_{y}-\frac{3k^{2}\mu}{4}\right)  \iota_{0}\cdot\tilde{\iota}=0
\]
can be verifies directly. Alternatively, one can use the bilinear identity
$\left(  \ref{back}\right)  $ to prove this.
\end{proof}

Since $\tilde{\iota}$ is of travelling wave type, we define a function $\iota$
through $\tilde{\iota}\left(  t,x,y\right)  =\iota\left(  x-t,y\right)  .$
Similarly as before, we need to investigate the linearized Backlund
transformation between $\iota_{0}$ and $\iota.$ That is%
\begin{equation}
\left\{
\begin{array}
[c]{l}%
\left(  D_{x}^{2}+\mu D_{x}+\frac{1}{\sqrt{3}}iD_{y}\right)  \phi\cdot
\iota-\lambda\phi\iota=\mathcal{G}_{1}\eta,\\
\left(  -D_{x}+3\lambda D_{x}-\sqrt{3}\mu iD_{y}+D_{x}^{3}-\sqrt{3}iD_{x}%
D_{y}-\frac{3k^{2}\mu}{4}\right)  \phi\cdot\iota=\mathcal{N}_{1}\eta.
\end{array}
\right.  \label{HB1}%
\end{equation}
Here
\[
\mathcal{G}_{1}\eta=-\left(  D_{x}^{2}+\mu D_{x}+\frac{1}{\sqrt{3}}%
iD_{y}\right)  \iota_{0}\cdot\eta+\lambda\iota_{0}\eta
\]
and
\[
\mathcal{N}_{1}\eta=-\left(  -D_{x}+3\lambda D_{x}-\sqrt{3}\mu iD_{y}%
+D_{x}^{3}-\sqrt{3}iD_{x}D_{y}-\frac{3k^{2}\mu}{4}\right)  \iota_{0}\cdot
\eta.
\]
Let us write the first equation as
\[
iD_{y}\phi\cdot\iota=-\left(  \sqrt{3}D_{x}^{2}+\sqrt{3}\mu D_{x}\right)
\phi\cdot\iota+\sqrt{3}\lambda\phi\iota+\sqrt{3}\mathcal{G}_{1}\eta.
\]
From this, we get
\begin{align*}
\partial_{y}\phi &  =\frac{1}{\iota}i\left[  \sqrt{3}\mu\left(  \partial
_{x}\phi\iota-\phi\partial_{x}\iota\right)  +\sqrt{3}\left(  \partial_{x}%
^{2}\phi\iota-2\partial_{x}\phi\partial_{x}\iota+\phi\partial_{x}^{2}%
\iota\right)  \right] \\
&  +\frac{\partial_{y}\iota}{\iota}\phi-\sqrt{3}\lambda i\phi-\frac{\sqrt{3}%
i}{\iota}\mathcal{G}_{1}\eta.
\end{align*}
Inserting $\left(  \ref{dy}\right)  $ into the right hand side of the second
equation, we get the following third order ODE:
\begin{align}
&  4\phi^{\prime\prime\prime}+\left(  6\mu-12\frac{\partial_{x}\iota}{\iota
}\right)  \phi^{\prime\prime}+\left(  3\mu^{2}-1-12\mu\frac{\partial_{x}\iota
}{\iota}+12\frac{\left(  \partial_{x}\iota\right)  ^{2}}{\iota^{2}}\right)
\phi^{\prime}\nonumber\\
&  =\frac{1}{\iota}\left(  3\partial_{x}\left(  \mathcal{G}_{1}\eta\right)
-\frac{6\partial_{x}\iota}{\iota}\mathcal{G}_{1}\eta+\mathcal{N}_{1}\eta
+3\mu\mathcal{G}_{1}\eta\right)  . \label{inh}%
\end{align}
Letting $g=\phi^{\prime},$ we obtain the corresponding homogeneous equation
\begin{align}
&  4g^{\prime\prime}+\left(  6\mu-6k\tanh\left(  \frac{k}{2}\left(
x-py\right)  \right)  \right)  g^{\prime}\nonumber\\
&  +\left(  -k^{2}-6\mu k\tanh\left(  \frac{k}{2}\left(  x-py\right)  \right)
+3k^{2}\tanh^{2}\left(  \frac{k}{2}\left(  x-py\right)  \right)  \right)  g=0.
\label{tau1}%
\end{align}

Next, we would like to find solutions for the equation $\left(  \ref{tau1}%
\right)  .$ Introduce the new variable $z=\tanh\frac{k\left(  x-py\right)
}{2}.$ The equation $\left(  \ref{tau1}\right)  $ becomes%
\begin{align}
&  k^{2}\left(  1-z^{2}\right)  ^{2}g_{zz}+\left(  \left(  3\mu-3kz\right)
k\left(  1-z^{2}\right)  -2k^{2}\left(  1-z^{2}\right)  z\right)
g_{z}\nonumber\\
&  +\left(  -k^{2}-6\mu kz+3k^{2}z^{2}\right)  g=0. \label{ode}%
\end{align}
Recall that $\mu=-\sqrt{\frac{1-k^{2}}{3}}.$

\begin{lemma}
\label{doe3}The equation $\left(  \ref{ode}\right)  $ has solutions of the
form
\[
\xi_{1}=\frac{z-\frac{3\mu}{k}}{z^{2}-1}\text{ \ and }\xi_{2}=\frac{1}%
{\sqrt{z^{2}-1}}\left(  \frac{z-1}{z+1}\right)  ^{\frac{3\mu}{2k}}.
\]

\end{lemma}

Here we have abused some notations. The function $\xi_{i}$ is not the same as
the one appeared in the previous sections.

One crucial fact is that as $\left\vert z\right\vert \rightarrow+\infty
,\xi_{1}=O\left(  \frac{1}{z}\right)  ,\xi_{2}=O\left(  \frac{1}{z}\right)  .$
This implies that at the singularities of $\left(  \ref{eta}\right)  ,$ where
$\iota=0,$ the solutions $\xi_{i}$ are actually smooth.

For any given function $\eta,$ with these explicit fundamental solutions, the
solutions of the inhomogeneous third order ODE $\left(  \ref{inh}\right)  $
can be written down using the variation of parameter formula, as have done for
the $\tau_{0},\tau_{1}$ case. Note that as $x\rightarrow+\infty,$
\[
\xi_{1}\sim e^{kx},\xi_{2}\sim e^{\frac{k-3\mu}{2}x}.
\]
while as $x\rightarrow-\infty,$
\[
\xi_{1}\sim e^{-kx},\xi_{2}\sim e^{-\frac{k+3\mu}{2}x}.
\]
The Wronskian of $\xi_{1}$ and $\xi_{2}$ behaves like $e^{-\frac{3k}{2}%
-\frac{3}{2}\mu}.$ We also remark that $\xi_{1}$ and $\xi_{2}$ are both
$\frac{2\pi}{kb}$-periodic in $y,$ while $\iota_{1}$ is only $\frac{4\pi}{kb}$-periodic.

Now we suppose that as $x\rightarrow+\infty$, the function $\phi=\rho\left(
y\right)  e^{\frac{k-3\mu}{2}x}$ satisfies asmptotically the equation
\begin{equation}
\left(  D_{x}^{2}+\mu D_{x}+\frac{1}{\sqrt{3}}iD_{y}\right)  \phi\cdot
\iota-\lambda\phi\iota=0. \label{first}%
\end{equation}
Let us write $e^{\frac{k-3\mu}{2}x}$ as $\zeta.$ Then the left hand side of
$\left(  \ref{first}\right)  $ is
\begin{align*}
&  \partial_{x}^{2}\zeta\iota\rho-2\partial_{x}\zeta\partial_{x}\iota
\rho+\zeta\partial_{x}^{2}\iota\rho+\mu\partial_{x}\zeta\iota\rho-\mu\zeta
\rho\partial_{x}\iota\\
&  +\frac{i}{\sqrt{3}}\left(  \partial_{y}\zeta\rho+\zeta\rho^{\prime}\right)
\iota-\frac{i}{\sqrt{3}}\zeta\rho\partial_{y}\iota-\lambda\zeta\rho\iota.
\end{align*}
For each fixed $y,$ sending $x\rightarrow+\infty\,$, using the asymptotic
behavior of $\iota$ and writing $\alpha=\frac{k-3\mu}{2},$ we get
\[
\frac{i}{\sqrt{3}}\rho^{\prime}+\left(  \alpha^{2}-\alpha k+\frac{k^{2}}%
{4}+\mu\alpha-\mu\frac{k}{2}-\lambda\right)  \rho=0.
\]
Hence $\rho=c\exp\left(  \frac{\sqrt{3}\left(  1-2k^{2}\right)  }{4}iy\right)
,$ where $c$ is a constant. Similarly, $\exp\left(  \frac{kb}{2}yi\right)
e^{kx}$ approximately solves $\left(  \ref{first}\right)  $ as $x\rightarrow
+\infty.$

\begin{remark}
We recall that the solutions of the homogeneous system in Lemma \ref{Tau2}
indeed have explicit formulas. However, we don't know whether explicit
formulas are still available for the homogeneous version of system $\left(
\ref{HB1}\right)  .$ In view of the fact that the minimal period of the
function $\exp\left(  \frac{\sqrt{3}\left(  1-2k^{2}\right)  }{4}iy\right)  $
is in general not equal to that of $\iota,$ we conjecture that no explicit
formula is available in this case. Note that the indefinite integral
\[
\int\left(  \cosh x\right)  \left(  \frac{\tanh x-1}{\tanh x+1}\right)
^{\frac{3\mu}{2k}}dx
\]
also seems to not been able to explicitly integrated.
\end{remark}

For later purpose, let us define the function
\[
\Phi_{0}:=\left(  D_{x}^{2}+\mu D_{x}+\frac{1}{\sqrt{3}}iD_{y}\right)
\phi\cdot\iota-\lambda\phi\iota+\left(  D_{x}^{2}+\mu D_{x}+\frac{1}{\sqrt{3}%
}iD_{y}\right)  \iota_{0}\cdot\eta-\lambda\iota_{0}\eta.
\]
We have the following

\begin{proposition}
\label{third order ode}Suppose $\phi$ satisfies $\left(  \ref{inh}\right)  $
and $\eta$ satisfies the linearized bilinear KP-I equation at $\iota.$ Then
\[
\partial_{x}^{3}\Phi_{0}=\left(  -\frac{3\mu}{2}+\frac{6}{\iota}\partial
_{x}\left(  \iota\right)  \right)  \partial_{x}^{2}\Phi_{0}+b\partial_{x}%
\Phi_{0}+c\Phi_{0},
\]
where%
\[
b=\frac{1}{4}k^{2}+6\mu\frac{\partial_{x}\iota}{\iota}+3\frac{\partial_{x}%
^{2}\iota}{\iota}-15\left(  \frac{\partial_{x}\iota}{\iota}\right)  ^{2},
\]%
\begin{align*}
c  &  =-\frac{1}{4}k^{2}\frac{\partial_{x}\iota}{\iota}-6\mu\left(
\frac{\partial_{x}\iota}{\iota}\right)  ^{2}+\frac{3\mu}{2}\frac{\partial
_{x}^{2}\iota}{\iota}+15\left(  \frac{\partial_{x}\iota}{\iota}\right)  ^{3}\\
&  -9\frac{\partial_{x}\iota\partial_{x}^{2}\iota}{\iota^{2}}+\frac
{\partial_{x}^{3}\iota}{\iota}.
\end{align*}

\end{proposition}

\begin{proof}
This is similar to Porposition \ref{p1} and can be directly verified with the
help of Maple.
\end{proof}

\subsection{The linearized B\"{a}cklund transformation between $\iota_{1}$ and
$\iota_{2}$}

Let $k\in\left(  0,\frac{1}{2}\right)  .$ Throughout the paper, we denote by
$A$ the constant $\sqrt{\frac{1-4k^{2}}{1-k^{2}}}.$ Then we define
\begin{align*}
\tilde{\iota}_{2}  &  =e^{\frac{\eta_{1}+\eta_{2}}{2}}+e^{-\frac{\eta_{1}%
+\eta_{2}}{2}}+A\left(  e^{\frac{\eta_{1}-\eta_{2}}{2}}+e^{\frac{\eta_{2}%
-\eta_{1}}{2}}\right) \\
&  =2\left(  \cosh\left(  k\left(  x-t\right)  \right)  +A\cos\left(
kby\right)  \right)  ,
\end{align*}
where $\eta_{1}=k\left(  x-py-t\right)  ,\eta_{2}=k\left(  x+py-t\right)  .$
As is well known, $\tilde{\iota}_{2}$ is a solution to the bilinear KP-I
equation, as can be checked by hand. Define $\iota_{2}$ through the relation
$\tilde{\iota}_{2}\left(  t,x,y\right)  =\iota_{2}\left(  x-t,y\right)  .$ Let
$\iota_{1}=r\exp(\frac{k\left(  x-py\right)  }{2})+\exp(-\frac{k\left(
x-py\right)  }{2})$ and $\tilde{\iota}_{1}=\iota_{1}\left(  x-t,y\right)  ,$
where
\[
\mu^{\ast}=-\frac{pi}{\sqrt{3}},r=\frac{\mu^{\ast}A}{k+\mu^{\ast}}.
\]
Note that $\iota_{1}$ is simply an $x$-translation of the function $\iota$
discussed in the previous section and hence the results proved there are also
true for $\iota_{1}.$ We emphasize that $\iota_{2}$ is $\frac{2\pi}{kb}%
$-periodic in $y$, while the minimal period of $\iota_{1}$ in the $y$
direction is equal to $\frac{4\pi}{kb}.$

\begin{lemma}
We have the following B\"{a}cklund transformation between $\tilde{\iota}_{1}$
and $\tilde{\iota}_{2}$:
\begin{equation}
\left\{
\begin{array}
[c]{l}%
\left(  D_{x}^{2}+\mu^{\ast}D_{x}+\frac{1}{\sqrt{3}}iD_{y}\right)
\tilde{\iota}_{1}\cdot\tilde{\iota}_{2}=\lambda\tilde{\iota}_{1}\tilde{\iota
}_{2},\\
\left(  D_{t}+3\lambda D_{x}-\sqrt{3}\mu^{\ast}iD_{y}+D_{x}^{3}-\sqrt{3}%
iD_{x}D_{y}-\frac{3k^{2}\mu^{\ast}}{4}\right)  \tilde{\iota}_{1}\cdot
\tilde{\iota}_{2}=0.
\end{array}
\right.  \label{iota1}%
\end{equation}

\end{lemma}

\begin{proof}
We have not been able to locate a reference in the literature for this result.
Therefore let us sketch the proof below.

We compute
\begin{align*}
&  \left(  D_{x}^{2}+\mu^{\ast}D_{x}+\frac{iD_{y}}{\sqrt{3}}\right)
\tilde{\iota}_{1}\cdot\tilde{\iota}_{2}-\lambda\tilde{\iota}_{1}\tilde{\iota
}_{2}\\
&  =\partial_{x}^{2}\tilde{\iota}_{1}\tilde{\iota}_{2}-2\partial_{x}%
\tilde{\iota}_{1}\partial_{x}\tilde{\iota}_{2}+\tilde{\iota}_{1}\partial
_{x}^{2}\tilde{\iota}_{2}\\
&  +\mu^{\ast}\left(  \partial_{x}\tilde{\iota}_{1}\iota_{2}-\tilde{\iota}%
_{1}\partial_{x}\tilde{\iota}_{2}\right)  +\frac{i}{\sqrt{3}}\left(
\partial_{y}\tilde{\iota}_{1}\tilde{\iota}_{2}-\tilde{\iota}_{1}\partial
_{y}\tilde{\iota}_{2}\right)  -\lambda\tilde{\iota}_{1}\tilde{\iota}_{2}.
\end{align*}
This is equal to%
\begin{align*}
&  -k^{2}\left(  -e^{-\frac{\eta_{1}}{2}}+re^{\frac{\eta_{1}}{2}}\right)
\left(  -e^{-\frac{\eta_{1}+\eta_{2}}{2}}+e^{\frac{\eta_{1}+\eta_{2}}{2}%
}\right)  +k^{2}\left(  e^{-\frac{\eta_{1}}{2}}+re^{\frac{\eta_{1}}{2}%
}\right)  \left(  e^{-\frac{\eta_{1}+\eta_{2}}{2}}+e^{\frac{\eta_{1}+\eta_{2}%
}{2}}\right) \\
&  -\frac{pi}{\sqrt{3}}\frac{k}{2}\left(  -e^{-\frac{\eta_{1}}{2}}%
+re^{\frac{\eta_{1}}{2}}\right)  \left(  e^{-\frac{\eta_{1}+\eta_{2}}{2}%
}+Ae^{\frac{\eta_{1}-\eta_{2}}{2}}+Ae^{\frac{-\eta_{1}+\eta_{2}}{2}}%
+e^{\frac{\eta_{1}+\eta_{2}}{2}}\right) \\
&  +\frac{pi}{\sqrt{3}}k\left(  e^{-\frac{\eta_{1}}{2}}+re^{\frac{\eta_{1}}%
{2}}\right)  \left(  -e^{-\frac{\eta_{1}+\eta_{2}}{2}}+e^{\frac{\eta_{1}%
+\eta_{2}}{2}}\right) \\
&  +\frac{i}{\sqrt{3}}\frac{pk}{2}\left(  \ e^{-\frac{\eta_{1}}{2}%
}-\ re^{\frac{\eta_{1}}{2}}\right)  \left(  e^{-\frac{\eta_{1}+\eta_{2}}{2}%
}+Ae^{\frac{\eta_{1}-\eta_{2}}{2}}+Ae^{\frac{-\eta_{1}+\eta_{2}}{2}}%
+e^{\frac{\eta_{1}+\eta_{2}}{2}}\right) \\
&  -\frac{ipk}{\sqrt{3}}\left(  e^{-\frac{\eta_{1}}{2}}+re^{\frac{\eta_{1}}%
{2}}\right)  \left(  -Ae^{\frac{\eta_{1}-\eta_{2}}{2}}+Ae^{\frac{-\eta
_{1}+\eta_{2}}{2}}\right)  .
\end{align*}
It can be simplified to
\begin{equation}
2k\left(  rk-\frac{\sqrt{3}ipr}{3}+\frac{\sqrt{3}ipA}{3}\right)
e^{-\frac{\eta_{2}}{2}}+2k\left(  k-\frac{\sqrt{3}i}{3}Apr+\frac{\sqrt{3}%
ip}{3}\right)  e^{\frac{\eta_{1}}{2}}. \label{rk}%
\end{equation}
Due to the choice of the constant $r,$ $\left(  \ref{rk}\right)  $ is equal to zero.

The second equation of $\left(  \ref{iota1}\right)  $ then follows from the
first one and the bilinear identity $\left(  \ref{back}\right)  .$
\end{proof}

Consider the linearized B\"{a}cklund transformation
\begin{equation}
\left\{
\begin{array}
[c]{l}%
\left(  D_{x}^{2}+\mu^{\ast}D_{x}+\frac{1}{\sqrt{3}}iD_{y}\right)  \phi
\cdot\iota_{2}-\lambda\phi\iota_{2}=\mathcal{G}_{2}\eta,\\
\left(  -D_{x}+3\lambda D_{x}-\sqrt{3}\mu^{\ast}iD_{y}+D_{x}^{3}-\sqrt
{3}iD_{x}D_{y}-\frac{3}{4}k^{2}\mu^{\ast}\right)  \phi\cdot\iota
_{2}=\mathcal{N}_{2}\eta.
\end{array}
\right.  \label{B2}%
\end{equation}
Here
\[
\mathcal{G}_{2}\eta:=-\left(  D_{x}^{2}+\mu^{\ast}D_{x}+\frac{1}{\sqrt{3}%
}iD_{y}\right)  \iota_{1}\cdot\eta+\lambda\iota_{1}\eta,
\]%
\[
\mathcal{N}_{2}\eta:=-\left(  -D_{x}+3\lambda D_{x}-\sqrt{3}\mu^{\ast}%
iD_{y}+D_{x}^{3}-\sqrt{3}iD_{x}D_{y}-\frac{3}{4}k^{2}\mu^{\ast}\right)
\iota_{1}\cdot\eta.
\]
Similarly as before, after we plug the $\partial_{y}\phi$ term into the first
equation of $\left(  \ref{B2}\right)  $ into the second one, we get the
following third order ODE:%
\begin{align}
&  4\iota_{2}\phi^{\prime\prime\prime}+\left(  6\mu^{\ast}\iota_{2}%
-12\partial_{x}\iota_{2}\right)  \phi^{\prime\prime}\nonumber\\
&  +\left(  -k^{2}\iota_{2}-12\mu^{\ast}\partial_{x}\iota_{2}+12\frac{\left(
\partial_{x}\iota_{2}\right)  ^{2}}{\iota_{2}}\right)  \phi^{\prime}%
+B\phi\nonumber\\
&  =3\partial_{x}\left(  \mathcal{G}_{2}\eta\right)  -\frac{6\partial_{x}%
\iota_{2}}{\iota_{2}}\mathcal{G}_{1}\eta+\mathcal{N}_{2}\eta+3\mu^{\ast
}\mathcal{G}_{2}\eta,\label{tau2}%
\end{align}
where the coefficient
\begin{align*}
B &  =2\partial_{x}^{3}\iota_{2}-6\frac{\partial_{x}\iota_{2}}{\iota_{2}%
}\partial_{x}^{2}\iota_{2}+k^{2}\partial_{x}\iota_{2}\\
&  +2\sqrt{3}i\frac{-\sqrt{3}i\mu^{\ast}\partial_{x}\iota_{2}+\partial
_{y}\iota_{2}}{\iota_{2}}\partial_{x}\iota_{2}\\
&  -2\sqrt{3}i\partial_{x}\partial_{y}\iota_{2}-6\mu^{\ast}\lambda\iota_{2}.
\end{align*}
Note that if we replace $\iota_{2},\mu^{\ast}$ by $\iota_{1},\mu$
respectively, $B$ will be $0,$ as we already computed.

The homogeneous version of $\left(  \ref{tau2}\right)  $ is
\begin{align}
&  4\iota_{2}\phi^{\prime\prime\prime}+\left(  6\mu^{\ast}\iota_{2}%
-12\partial_{x}\iota_{2}\right)  \phi^{\prime\prime}\nonumber\\
&  +\left(  -k^{2}\iota_{2}-12\mu^{\ast}\partial_{x}\iota_{2}+12\frac{\left(
\partial_{x}\iota_{2}\right)  ^{2}}{\iota_{2}}\right)  \phi^{\prime}%
+B\phi\nonumber\\
&  =0. \label{t2}%
\end{align}
Observe that $\iota_{1}$ is automatically a solution of this equation. We
write $\phi=\iota_{1}h,$ and $g=h^{\prime}.$ Then $g$ should satisfy%
\begin{align}
&  g^{\prime\prime}+\left(  3\frac{\iota_{1}^{\prime}}{\iota_{1}}+\frac
{3\mu^{\ast}}{2}-3\frac{\partial_{x}\iota_{2}}{\iota_{2}}\right)  g^{\prime
}\nonumber\\
&  +\left(  \frac{k^{2}}{2}+\left(  3\mu^{\ast}-6\frac{\partial_{x}\iota_{2}%
}{\iota_{2}}\right)  \frac{\iota_{1}^{\prime}}{\iota_{1}}-3\mu^{\ast}%
\frac{\partial_{x}\iota_{2}}{\iota_{2}}+3\frac{\left(  \partial_{x}\iota
_{2}\right)  ^{2}}{\iota_{2}^{2}}\right)  g=0. \label{eta}%
\end{align}
If we introduce the new variable $z=e^{k\left(  x-py\right)  }.$ Then we get
the following second order ODE for $g$:%
\begin{align*}
&  k^{2}z^{2}\frac{d^{2}g}{dz^{2}}+\left(  k^{2}z+\left(  3\frac{\partial
_{x}\tau_{1}}{\tau_{1}}+\frac{3\mu^{\ast}}{2}-3\frac{\partial_{x}\tau_{2}%
}{\tau_{2}}\right)  kz\right)  \frac{dg}{dz}\\
&  +\left(  \frac{k^{2}}{2}+\left(  3\mu^{\ast}-6\frac{\partial_{x}\tau_{2}%
}{\tau_{2}}\right)  \frac{\partial_{x}\tau_{1}}{\tau_{1}}-3\mu^{\ast}%
\frac{\partial_{x}\tau_{2}}{\tau_{2}}+3\frac{\left(  \partial_{x}\tau
_{2}\right)  ^{2}}{\tau_{2}^{2}}\right)  g\\
&  =0.
\end{align*}
To have an idea for what happened in this equation, we consider the special
case of $y=0.$ In this case, the equation has the form
\begin{align*}
&  z^{2}\frac{d^{2}g}{dz^{2}}+z\left(  1+\frac{3}{2}\frac{rz-1}{rz+1}%
+\frac{3\mu^{\ast}}{2k}-3\frac{z^{2}-1}{z^{2}+2Az+1}\right)  \frac{dg}{dz}\\
&  +\left(  \frac{1}{2}+\left(  3\frac{\mu^{\ast}}{k}-6\frac{z^{2}-1}%
{z^{2}+2Az+1}\right)  \frac{1}{2}\frac{rz-1}{rz+1}-3\frac{\mu^{\ast}}{k^{2}%
}\frac{z^{2}-1}{z^{2}+2Az+1}+\frac{3}{k^{2}}\left(  \frac{z^{2}-1}%
{z^{2}+2Az+1}\right)  ^{2}\right)  g\\
&  =0.
\end{align*}
If we choose $k=\frac{1}{2},$ then $\mu^{\ast}=\frac{1}{2},$ $A=\sqrt
{\frac{1-4k^{2}}{1-k^{2}}}=0,r=0.$ Hence the equation reduces to
\begin{equation}
\frac{d^{2}}{dz^{2}}g+\frac{4-2z^{2}}{\left(  z^{2}+1\right)  z}\frac{d}%
{dz}g+\frac{2\left(  4z^{4}-13z^{2}+7\right)  }{z^{2}\left(  z^{2}+1\right)
^{2}}g=0. \label{equag}%
\end{equation}
The solutions of $\left(  \ref{equag}\right)  $ are given by hypergeometric
functions. For general $k\in\left(  0,\frac{1}{2}\right)  ,$ it seems to us
that the solutions of $\left(  \ref{equag}\right)  $ do not have explicit
expressions. This contrasts with the previous cases and is an interesing issue.

The next lemma gives the asymptotic behavior at infinity of solutions to
$\left(  \ref{eta}\right)  $. Recall that $\mu^{\ast}=\frac{b}{\sqrt{3}}.$

\begin{lemma}
\label{Beta}The indicial roots of $\left(  \ref{eta}\right)  $ are given by
\[
\lambda_{1}^{+}=k,\lambda_{2}^{+}=\frac{k}{2}-\frac{3}{2}\mu^{\ast},
\]
and%
\[
\lambda_{1}^{-}=-k,\lambda_{2}^{-}=-\frac{k}{2}-\frac{3}{2}\mu^{\ast}.
\]
Consequently, the indicial roots of $\left(  \ref{t2}\right)  $ are
\begin{align*}
\alpha_{0}^{+}  &  =\frac{k}{2},\alpha_{1}^{+}=\frac{3k}{2},\alpha_{2}%
^{+}=k-\frac{3}{2}\mu^{\ast},\\
\alpha_{0}^{-}  &  =-\frac{k}{2},\alpha_{1}^{-}=-\frac{3k}{2},\alpha_{2}%
^{-}=-k-\frac{3}{2}\mu^{\ast}.
\end{align*}

\end{lemma}

\begin{proof}
Recall that as $x\rightarrow\pm\infty,\frac{\iota_{1}^{\prime}}{\iota_{1}%
}\rightarrow\pm\frac{k}{2}.$ On the other hand,
\[
\iota_{2}=2\left(  \cosh\left(  kx\right)  +A\cosh\left(  kbiy\right)
\right)  .
\]
Therefore,%
\[
\frac{\partial_{x}\iota_{2}}{\iota_{2}}=\frac{k\sinh\left(  kx\right)  }%
{\cosh\left(  kx\right)  +A\cosh\left(  kbiy\right)  }.
\]
From this we see that
\[
\frac{\partial_{x}\iota_{2}}{\iota_{2}}\rightarrow\pm k,\text{ as
}x\rightarrow\pm\infty.
\]
It follows that the limit equation of $\left(  \ref{eta}\right)  $ at
$+\infty$ is
\[
\eta^{\prime\prime}+\left(  \frac{3\mu^{\ast}}{2}-\frac{3}{2}k\right)
\eta^{\prime}+\left(  \frac{k^{2}}{2}-\frac{3}{2}\mu^{\ast}k\right)  \eta=0.
\]
Hence the indicial roots are
\[
\lambda_{1,2}^{+}=\frac{-\left(  \frac{3\mu^{\ast}}{2}-\frac{3}{2}k\right)
\pm\sqrt{\left(  \frac{3\mu^{\ast}}{2}-\frac{3}{2}k\right)  ^{2}-4\left(
\frac{k^{2}}{2}-\frac{3}{2}\mu^{\ast}k\right)  }}{2}.
\]
Using the fact that $\mu^{\ast}=\frac{b}{\sqrt{3}},$ we get $\lambda_{1}%
^{+}=k,\lambda_{2}^{+}=\frac{k}{2}-\frac{3}{2}\mu^{\ast}.$

Similarly, as $x\rightarrow-\infty,$ we get
\[
\eta^{\prime\prime}+\left(  \frac{3\mu^{\ast}}{2}+\frac{3}{2}k\right)
\eta^{\prime}+\left(  \frac{k^{2}}{2}+\frac{3}{2}\mu^{\ast}k\right)  \eta=0.
\]
Hence the indicial roots are
\[
\frac{-\left(  \frac{3\mu^{\ast}}{2}+\frac{3}{2}k\right)  \pm\sqrt{\left(
\frac{3\mu^{\ast}}{2}+\frac{3}{2}k\right)  ^{2}-4\left(  \frac{k^{2}}{2}%
+\frac{3}{2}\mu^{\ast}k\right)  }}{2}.
\]
Hence $\lambda_{1}^{-}=-k,\lambda_{2}^{-}=-\frac{k}{2}-\frac{3}{2}\mu^{\ast}.$
The proof is thus completed.
\end{proof}

For each fixed $y,$ we use $\beta_{j},j=1,2,$ to denote the fundamental
solutions of $\left(  \ref{t2}\right)  $ with the asymptotic behavior
$e^{\alpha_{i}^{-}x}$ as $x\rightarrow-\infty.$ Since we don't have explicit
formulas for the solutions of $\left(  \ref{t2}\right)  ,$ in the sequel, we
would like to analyze the existence and asymptotic behavior of the solutions
to the system%
\begin{equation}
\left\{
\begin{array}
[c]{l}%
\left(  D_{x}^{2}+\mu^{\ast}D_{x}+\frac{1}{\sqrt{3}}iD_{y}\right)  \phi
\cdot\iota_{2}-\lambda\phi\iota_{2}=0,\\
\left(  -D_{x}+3\lambda D_{x}-\sqrt{3}\mu^{\ast}iD_{y}+D_{x}^{3}-\sqrt
{3}iD_{x}D_{y}-\frac{3}{4}k^{2}\mu^{\ast}\right)  \phi\cdot\iota_{2}=0.
\end{array}
\right.  \label{hB2}%
\end{equation}

Suppose a function of the form $\rho\left(  y\right)  \beta_{1}\left(
x,y\right)  ,$ with the initial condition $\rho\left(  0\right)  =1$
satisfying $\left(  \ref{hB2}\right)  $ approximately, as $x\rightarrow
-\infty.$ Inserting it into the left hand side of the first equation of the
system $\left(  \ref{hB2}\right)  $, we get%
\begin{align*}
&  \partial_{x}^{2}\beta_{1}\iota_{2}\rho-2\partial_{x}\beta_{1}\partial
_{x}\iota_{2}\rho+\beta_{1}\partial_{x}^{2}\iota_{2}\rho+\mu^{\ast}%
\partial_{x}\beta_{1}\iota_{2}\rho-\mu^{\ast}\beta_{1}\rho\partial_{x}%
\iota_{2}\\
&  +\frac{i}{\sqrt{3}}\left(  \partial_{y}\beta_{1}\rho+\beta_{1}\rho^{\prime
}\right)  \iota_{2}-\frac{i}{\sqrt{3}}\beta_{1}\rho\partial_{y}\iota
_{2}-\lambda\beta_{1}\rho\iota_{2}.
\end{align*}
For each fixed $y,$ sending $x\rightarrow-\infty\,$, using the asymptotic
behavior of $\beta_{1}$ and $\iota_{2},$ we get
\[
\frac{i}{\sqrt{3}}\rho^{\prime}+\left(  \left(  \alpha_{1}^{-}\right)
^{2}+2\alpha_{1}^{-}k+k^{2}+\mu^{\ast}\alpha_{1}^{-}+\mu^{\ast}k-\lambda
\right)  \rho=0.
\]
Note that $\alpha_{1}^{-}=-\frac{3}{2}k.$ Hence $\rho\left(  y\right)
=\exp\left(  -\frac{k\sqrt{1-k^{2}}}{2}iy\right)  .$

Next we define the function
\[
\Phi_{0}^{\ast}:=\left(  D_{x}^{2}+\mu^{\ast}D_{x}+\frac{1}{\sqrt{3}}%
iD_{y}\right)  \phi\cdot\iota_{2}-\lambda\phi\iota_{1}+\left(  D_{x}^{2}%
+\mu^{\ast}D_{x}+\frac{1}{\sqrt{3}}iD_{y}\right)  \iota_{0}\cdot\eta
-\lambda\iota_{0}\eta.
\]
A result parallel to Proposition \ref{third order ode} is:

\begin{proposition}
\label{P5}Suppose $\phi$ satisfies $\left(  \ref{tau2}\right)  $ and $\eta$
satisfies the linearized bilinear KP-I equation at $\iota_{2}.$ Then
\[
\partial_{x}^{3}\Phi_{0}^{\ast}=\left(  -\frac{3\mu^{\ast}}{2}+\frac{6}%
{\iota_{2}}\partial_{x}\left(  \iota_{2}\right)  \right)  \partial_{x}^{2}%
\Phi_{0}^{\ast}+b\partial_{x}\Phi_{0}^{\ast}+c\Phi_{0}^{\ast},
\]
where%
\[
b=\frac{1}{4}k^{2}+6\mu^{\ast}\frac{\partial_{x}\iota_{2}}{\iota_{2}}%
+3\frac{\partial_{x}^{2}\iota_{2}}{\iota_{2}}-15\left(  \frac{\partial
_{x}\iota_{2}}{\iota_{2}}\right)  ^{2},
\]%
\begin{align*}
c  &  =-\frac{15\mu^{\ast}}{2}\left(  \frac{\partial_{x}\iota_{2}}{\iota_{2}%
}\right)  ^{2}-\frac{1}{2}k^{2}\frac{\partial_{x}\iota_{2}}{\iota_{2}}%
+\frac{3}{2}\mu^{\ast}\lambda+\frac{\sqrt{3}i}{2}\frac{\partial_{x}%
\partial_{y}\iota_{2}}{\iota_{2}}\\
&  -\frac{\sqrt{3}i}{2\iota_{2}^{2}}\partial_{x}\iota_{2}\partial_{y}\iota
_{2}+\frac{3\mu^{\ast}}{2}\frac{\partial_{x}^{2}\iota_{2}}{\iota_{2}}+\frac
{1}{2}\frac{\partial_{x}^{3}\iota_{2}}{\iota_{2}}\\
&  -\frac{15}{2}\frac{\partial_{x}\iota_{2}\partial_{x}^{2}\iota_{2}}%
{\iota_{2}^{2}}+15\left(  \frac{\partial_{x}\iota_{2}}{\iota_{2}}\right)
^{3}.
\end{align*}

\end{proposition}

\begin{remark}
The coefficients before $\partial_{x}^{2}\Phi_{0}^{\ast}$ and $\partial
_{x}\Phi_{0}^{\ast}$ are essentially same as that of $\left(
\ref{third order ode}\right)  .$ The coefficient before $\Phi_{0}^{\ast}$ is
different, due to the fact that the coefficient $B$ in the equation $\left(
\ref{t2}\right)  $ is nonzero in this case. Indeed, Proposition \ref{p1},
Proposition \ref{P4}, and Proposition \ref{third order ode} can regarded as
special cases of this identity, with different choices of the parameters
$\mu,\lambda,\iota.$
\end{remark}

\subsection{Proof of the nondegeneracy of the periodic solutions}

Recall that $\iota_{2}$ is periodic in $y$ with period $\frac{2\pi}{kb}.$ When
$k=\frac{1}{2},$ $\frac{2\pi}{kb}=\frac{8\pi}{\sqrt{3}}.$ As $k\rightarrow0,$
the period tends to $+\infty.$ The KP-I equation has a family of solutions
$2\partial_{x}^{2}\ln\tilde{\iota}_{2},$ which converges to the lump solution
as $k\rightarrow0.$ In this part, we prove that $T_{2}=2\partial_{x}^{2}%
\ln\iota_{2}$ is nondegenerated in the class of $\frac{2\pi}{kb}$-periodic(in
$y$) solutions.

\begin{lemma}
\label{kernel}Let $k\in\left(  0,\frac{1}{2}\right)  $ be fixed. Let $\varphi$
be a function $\frac{2\pi}{kb}$-periodic in $y.$ Suppose for each fixed $y,$%
\[
\varphi\left(  x,y\right)  \rightarrow0,\text{ as }\left\vert x\right\vert
\rightarrow+\infty.
\]
Assume $\varphi$ solves the linearized KP-I equation at $T_{2}:$
\[
\partial_{x}^{2}\left(  \partial_{x}^{2}\varphi-\varphi+6T_{2}\varphi\right)
-\partial_{y}^{2}\varphi=0.
\]
Then $\varphi$ has the estimate:%
\[
\left\vert \varphi\left(  x,y\right)  \right\vert \leq Ce^{-k\left\vert
x\right\vert }.
\]

\end{lemma}

\begin{proof}
Since $\varphi$ is $\frac{2\pi}{kb}$-periodic in $y,$ we can write it in the
form of a Fourier series. Suppose
\[
\varphi=Z_{0}\left(  x\right)  +\sum_{n=1}^{+\infty}\left(  Z_{n}\left(
x\right)  \cos\left(  kbny\right)  +Z_{n}^{\ast}\left(  x\right)  \sin\left(
kbny\right)  \right)  .
\]
Then $Z_{n}\left(  x\right)  $ satisfies
\[
\partial_{x}^{2}\left(  \partial_{x}^{2}Z_{n}-Z_{n}+6T_{2}Z_{n}\right)
+k^{2}b^{2}n^{2}Z_{n}=0.
\]
Note that the indicial roots $\lambda$ of the equation
\[
Z^{\left(  4\right)  }-Z^{\left(  2\right)  }+k^{2}b^{2}n^{2}Z=0
\]
are given by
\[
\lambda^{2}=\frac{1\pm\sqrt{1-4k^{2}b^{2}n^{2}}}{2}.
\]
We obtain
\[
\lambda=\pm\sqrt{\frac{1\pm\sqrt{1-4k^{2}b^{2}n^{2}}}{2}}.
\]
We claim that for $n\in\mathbb{N},$ $\left\vert \operatorname{Re}%
\lambda\right\vert >k.$ There are two cases.

Case 1: $4k^{2}b^{2}n^{2}<1.$

We need to show
\[
\frac{1-\sqrt{1-4k^{2}b^{2}n^{2}}}{2}\geq k^{2}.
\]
The validity of this inequality holds follows from the fact that
\[
1-4k^{2}b^{2}=\left(  1-2k^{2}\right)  ^{2}.
\]

Case 2: $4k^{2}b^{2}n^{2}>1.$

Writing $\lambda=c+di,$ we have, for some $\alpha\in\mathbb{R},$
\[
c^{2}-d^{2}+2cdi=\lambda^{2}=\frac{1\pm\alpha i}{2}.
\]
This implies $c^{2}-d^{2}=\frac{1}{2}$ and $2cd=\frac{\alpha}{2}.$ Since
$k\leq\frac{1}{2},$ we find that $\left\vert c\right\vert \geq\frac{1}%
{\sqrt{2}}>k.$ This finishes the proof.
\end{proof}

Let $\varphi$ be a function satisfying the assumption Lemma \ref{kernel}. We
differentiate the equation $T_{2}=2\partial_{x}^{2}\ln\iota_{2}$ and get a
function $\eta$
\[
\varphi=2\partial_{x}^{2}\frac{\eta}{\iota_{2}}.
\]
Hence we get a corresponding solution $\eta$ of the linearized bilinear KP-I
equation, with the form
\[
\eta=\iota_{2}\partial_{x}^{-2}\frac{\varphi}{2}.
\]
Here $\partial_{x}^{-1}:=\int_{-\infty}^{x}.$ Observe that as $\left\vert
x\right\vert \rightarrow\infty$, $\iota_{2}=O\left(  e^{k\left\vert
x\right\vert }\right)  .$ Due to the estimate of $\varphi,$ we conclude that
$\eta\leq C,$ for $x\leq0.$ The next step is to solve the linearized
B\"{a}cklund transformation between $\iota_{1},\iota_{2}.$

\begin{lemma}
\label{Lemma1}\bigskip Suppose $\eta$ is a function solving the linearized
bilinear KP-I equation at $\iota_{2}:$%
\[
-D_{x}^{2}\eta\cdot\iota_{2}+D_{x}^{4}\eta\cdot\iota_{2}=D_{y}^{2}\eta
\cdot\iota_{2}.
\]
Assume $\eta$ is $\frac{2\pi}{kb}$-periodic in $y$ and satisfies $\left\vert
\eta\left(  x,y\right)  \right\vert \leq\left(  1+\left\vert x\right\vert
\right)  e^{k\left\vert x\right\vert },$ and
\[
\left\vert \eta\left(  x,y\right)  \right\vert \leq C,\text{ for }x\leq0.
\]
Then the system $\left(  \ref{B2}\right)  $ has a solution $\phi,$ $\frac
{4\pi}{kb}$-periodic in $y,$ with
\[
\left\vert \phi\left(  x,y\right)  \right\vert \leq C,\text{ for }x\leq0.
\]
Moreover, for $x$ large, $\phi$ can be written as a linear combination of
$e^{\frac{3}{2}kx}$ and a function $\phi_{0}^{\ast}$ with
\[
\left\vert \phi_{0}^{\ast}\right\vert \leq C\left(  1+\left\vert x\right\vert
\right)  e^{k\left\vert x\right\vert }.
\]
Additionally, $\phi$ satisfies the linearized bilinear KP-I equation at
$\iota_{1}$%
\[
-D_{x}^{2}\phi\cdot\iota_{1}+D_{x}^{4}\phi\cdot\iota_{1}=D_{y}^{2}\phi
\cdot\iota_{1}.
\]

\end{lemma}

\begin{proof}
The proof is similar to the case of the linearized Backlund transformation
between $\tau_{1},\tau_{2}.$ The only difference is the analysis of the
asymptotic behavior at infinity, which follows from the asymptotic behavior of
the fundamental solutions $\beta_{1},\beta_{2},$ and $\iota_{1}$, of the
homogenous equation of the third order ODE $\left(  \ref{eta}\right)  ,$ given
by Lemma \ref{Beta}. We use $W$ to denote the Wronskian of $\beta_{1}%
,\beta_{2}.$ As $x\rightarrow-\infty,$ $W$ behaves like $e^{-\frac{3}%
{2}\left(  k+\mu^{\ast}\right)  }.$

For each fixed $y,$ the inhomogeneous equation we need to solve is
\begin{align*}
&  4\iota_{2}\phi^{\prime\prime\prime}+\left(  6\mu^{\ast}\iota_{2}%
-12\partial_{x}\iota_{2}\right)  \phi^{\prime\prime}+\left(  \left(
3\mu^{\ast2}-1\right)  \iota_{2}-12\mu^{\ast}\partial_{x}\iota_{2}%
+12\frac{\left(  \partial_{x}\iota_{2}\right)  ^{2}}{\iota_{2}}\right)
\phi^{\prime}+B\phi\\
&  =\left(  \mathcal{N}_{2}\eta+\frac{6\partial_{x}\tau_{2}}{\tau_{2}%
}\mathcal{G}_{2}\eta+3\partial_{x}\left(  \mathcal{G}_{2}\eta\right)
+3\mu\mathcal{G}_{2}\eta\right)  .
\end{align*}
Let us define
\[
M:=\frac{1}{4\iota_{2}}\left(  \mathcal{N}_{2}\eta+\frac{6\partial_{x}\tau
_{2}}{\tau_{2}}\mathcal{G}_{2}\eta+3\partial_{x}\left(  \mathcal{G}_{2}%
\eta\right)  +3\mu\mathcal{G}_{2}\eta\right)  .
\]
This equation has a solution of the form%
\begin{equation}
\phi_{0}\left(  x,y\right)  :=\iota_{1}\int_{-\infty}^{x}\left[  \beta_{2}%
\int_{-\infty}^{s}\left(  \frac{\beta_{1}}{W}M\right)  -\beta_{1}\int%
_{-\infty}^{s}\left(  \frac{\beta_{2}}{W}M\right)  \right]  ds. \label{f0}%
\end{equation}
Using the estimate of $\eta$, we obtain%
\[
\left\vert \mathcal{N}_{2}\eta\right\vert \leq Ce^{\frac{k}{2}\left\vert
x\right\vert }\left(  1+\left\vert x\right\vert \right)  e^{k\left\vert
x\right\vert },
\]
and for $x\leq0,$ $\left\vert \mathcal{N}_{2}\eta\right\vert \leq Ce^{\frac
{k}{2}\left\vert x\right\vert }.$ The term $\mathcal{G}_{2}\eta$ satisfies
similar estimates. From these bounds and $\left(  \ref{f0}\right)  ,$ we infer
that for $x\leq0,$
\[
\left\vert \phi_{0}\right\vert \leq C.
\]

The next step is to find a solution $\phi$ satisfying the first equation of
$\left(  \ref{B2}\right)  $ \ on one line parallel to the $y$ axis, say $x=0.$
We seek a solution of the form
\[
\phi=\phi_{0}+\rho\left(  y\right)  \iota_{1}+\rho_{1}\left(  y\right)
\beta_{1}+\rho_{2}\left(  y\right)  \beta_{2},
\]
with the initial condition $\rho\left(  0\right)  =\rho_{1}\left(  0\right)
=\rho_{2}\left(  0\right)  =0.$ Applying Proposition \ref{P5}, we conclude
that $\phi$ solves the whole system $\left(  \ref{B2}\right)  .$ Using the
asymptotic behavior of $\iota_{1},\beta_{1},\beta_{2}$ at $-\infty$(observe
that $\alpha_{1}^{-}=-\frac{3}{2}k<0$ and $\alpha_{2}^{-}=-k-\frac{3}{2}%
\mu^{\ast}<0$), we infer that $\rho_{i}\left(  y\right)  =0,i=1,2,3.$
Moreover, for $x$ positive and large, in view of the growth rate of $\beta
_{1},\beta_{2},\iota_{1},$ as $x\rightarrow+\infty,$ $\phi_{0}$ can be written
as a linear combination of $e^{\frac{3}{2}kx}$ and a function $\phi_{0}^{\ast
}$ satisfying
\begin{equation}
\left\vert \phi_{0}^{\ast}\right\vert \leq C\left(  1+\left\vert x\right\vert
\right)  e^{k\left\vert x\right\vert }. \label{fi0}%
\end{equation}

Finally we remark that the function $\phi_{0}$ is $\frac{4\pi}{kb}$-periodic
in $y,$ since $\eta$ is $\frac{2\pi}{kb}$-periodic and $\iota_{1}$ is
$\frac{4\pi}{kb}$-periodic.
\end{proof}

We also need to solve the linearized B\"{a}cklund transformation between
$\iota_{0}$ and $\iota_{1}.$ A parallel result of Lemma \ref{Lemma1} is

\begin{lemma}
\label{Lemma2}\bigskip Suppose $\phi$ is a function solving the linearized
bilinear KP-I equation at $\iota_{1}$ obtained from Lemma \ref{Lemma1}. Let
$\xi_{1},\xi_{2}$ be the functions defined in Lemma \ref{doe3}. Then the
system
\[
\left\{
\begin{array}
[c]{l}%
\left(  D_{x}^{2}+\mu D_{x}+\frac{1}{\sqrt{3}}iD_{y}\right)  \xi\cdot\iota
_{1}-\lambda\xi\iota_{1}=\mathcal{G}_{1}\phi,\\
\left(  -D_{x}+3\lambda D_{x}-\sqrt{3}\mu iD_{y}+D_{x}^{3}-\sqrt{3}iD_{x}%
D_{y}-\frac{3k^{2}\mu}{4}\right)  \xi\cdot\iota_{1}=\mathcal{N}_{1}\phi,
\end{array}
\right.
\]
has a solution $\xi,$ $\frac{4\pi}{kb}$-periodic in $y,$ with
\begin{equation}
\left\vert \xi\left(  x,y\right)  -C_{1}\left(  y\right)  e^{\frac{3}{2}%
kx}-C_{2}e^{\frac{k-3\mu}{2}x+\frac{\sqrt{3}\left(  1-2k^{2}\right)  }{4}%
iy}\right\vert \leq C\left(  1+x^{2}\right)  e^{k\left\vert x\right\vert },
\label{zi}%
\end{equation}
for some constants $C_{1}\left(  y\right)  ,C_{2},$ where $C_{2}$ is
independent of $y,$ and
\begin{equation}
\left\vert \xi\left(  x,y\right)  \right\vert \leq e^{-\frac{k}{2}\left\vert
x\right\vert },\text{ for }x\leq0. \label{negative}%
\end{equation}
Moreover $\xi$ satisfies the linearized KP-I equation at $\iota_{0}$:
\[
-D_{x}^{2}\xi\cdot\iota_{0}+D_{x}^{4}\xi\cdot\iota_{0}=D_{y}^{2}\xi\cdot
\iota_{0}.
\]

\end{lemma}

\begin{proof}
This follows from same arguments as in Lemma \ref{Lemma1}, using variation of
parameter formula. We omit the details.
\end{proof}

In the next step, for each given function $\xi,$ we would like to solve the
reversed linearized Backlund transformation system
\begin{equation}
\left\{
\begin{array}
[c]{l}%
\left(  D_{x}^{2}+\mu D_{x}+\frac{1}{\sqrt{3}}iD_{y}\right)  \iota_{0}%
\cdot\eta-\lambda\iota_{0}\eta=-P\xi,\\
\left(  -D_{x}+3\lambda D_{x}-\sqrt{3}\mu iD_{y}+D_{x}^{3}-\sqrt{3}iD_{x}%
D_{y}-\frac{3k^{2}\mu}{4}\right)  \iota_{0}\cdot\eta=-S\xi.
\end{array}
\right.  \label{reverse}%
\end{equation}
where
\begin{equation}
\left\{
\begin{array}
[c]{l}%
P\xi=\left(  D_{x}^{2}+\mu D_{x}+\frac{1}{\sqrt{3}}iD_{y}\right)  \xi
\cdot\iota_{1}+\lambda\xi\iota_{1},\\
S\xi=\left(  \left(  3\lambda-1\right)  D_{x}-\sqrt{3}\mu iD_{y}+D_{x}%
^{3}-\sqrt{3}iD_{x}D_{y}-\frac{3k^{2}\mu}{4}\right)  \xi\cdot\iota_{1}.
\end{array}
\right.  \label{ng}%
\end{equation}
The first equation of $\left(  \ref{reverse}\right)  $ can be written as
\[
\partial_{x}^{2}\eta-\mu\partial_{x}\eta+\frac{i}{\sqrt{3}}\left(
-\partial_{y}\eta\right)  -\lambda\eta=-P\xi.
\]
Inserting this into the second equation, we get%
\[
-4\partial_{x}^{3}\eta+6\mu\partial_{x}^{2}\eta+k^{2}\partial_{x}\eta
-3\mu\lambda\eta-\frac{3k^{2}\mu}{4}\eta+3\mu P\xi-3\partial_{x}\left(
P\xi\right)  +S\xi=0.
\]
The characteristic equation of the ODE%
\begin{equation}
-4\partial_{x}^{3}\eta+6\mu\partial_{x}^{2}\eta+k^{2}\partial_{x}\eta
-3\mu\lambda\eta-\frac{3k^{2}\mu}{4}\eta=0 \label{ODE}%
\end{equation}
is
\[
-4t^{3}-2\sqrt{3}\sqrt{1-k^{2}}t^{2}+k^{2}t+\frac{\sqrt{3}k^{2}\sqrt{1-k^{2}}%
}{2}=0.
\]
Hence the indicial roots of the equation $\left(  \ref{ODE}\right)  $ are
$-\frac{1}{2}k,\frac{1}{2}k,\frac{3\mu}{2}.$

\begin{lemma}
\label{NG}Let $P\xi,S\xi$ be defined in $\left(  \ref{ng}\right)  $. If
$\xi=e^{k\left(  x-py\right)  }.$ Then
\begin{align*}
&  3\mu P\xi-3\partial_{x}\left(  P\xi\right)  +S\xi\\
&  =-\frac{k^{2}}{8}\left(  3ke^{\frac{3}{2}\left(  k\left(  x-py\right)
\right)  }+\left(  9k-96\mu\right)  e^{\frac{1}{2}\left(  k\left(
x-py\right)  \right)  }\right)  .
\end{align*}
On the other hand, if $\xi=e^{k\left(  x+py\right)  },$ then
\begin{align*}
&  3\mu P\xi-3\partial_{x}\left(  P\xi\right)  +S\xi\\
&  =-\frac{k^{2}}{8}\left(  \left(  3k+96\mu\right)  e^{\frac{k}{2}\left(
3x+py\right)  }+9ke^{\frac{k}{2}\left(  x+3py\right)  }\right)  .
\end{align*}

\end{lemma}

\begin{proof}
Taking derivatives in $x$ for $P\xi,$ we get%
\begin{align*}
\partial_{x}\left(  P\xi\right)   &  =\iota_{1}\partial_{x}^{3}\xi+\left(
-\partial_{x}\iota_{1}+\mu\iota_{1}\right)  \partial_{x}^{2}\xi+\frac{i}%
{\sqrt{3}}\iota_{1}\partial_{x}\partial_{y}\xi+\frac{i}{\sqrt{3}}\partial
_{x}\iota_{1}\partial_{y}\xi\\
&  +\left(  -\partial_{x}^{2}\iota_{1}-\frac{i}{\sqrt{3}}\partial_{y}\iota
_{1}-\lambda\iota_{1}\right)  \partial_{x}\xi\\
&  +\left(  \partial_{x}^{3}\iota_{1}-\mu\partial_{x}^{2}\iota_{1}-\frac
{i}{\sqrt{3}}\partial_{x}\partial_{y}\iota_{1}-\lambda\partial_{x}\iota
_{1}\right)  \xi.
\end{align*}
Therefore, $-3\mu P\xi+3\partial_{x}\left(  P\xi\right)  -S\xi$ is equal to
\begin{align*}
2\iota_{1}\partial_{x}^{3}\xi &  +\left(  2\sqrt{3}i\iota_{1}\right)
\partial_{x}\partial_{y}\xi+\left(  -6\partial_{x}^{2}\iota_{1}+6\mu
\partial_{x}\iota_{1}+\lambda\iota_{1}-2\sqrt{3}i\partial_{y}\iota_{1}\right)
\partial_{x}\xi\\
&  +\left(  4\partial_{x}^{3}\iota_{1}-6\mu\partial_{x}^{2}\iota_{1}%
-\frac{7k^{2}}{4}\partial_{x}\iota_{1}+\frac{3k^{2}}{2}\mu\iota_{1}\right)
\xi.
\end{align*}
Direct substitution of explicit formula of $\xi$ into this expression gives us
the desired results.
\end{proof}

\begin{remark}
We already know that $\exp\left(  \frac{kb}{2}yi\right)  e^{kx}$ approximately
solves the equation $\left(  \ref{first}\right)  .$ This actually already
implies that in Lemma $\ref{NG},$ the $O\left(  e^{\frac{3}{2}kx}\right)  $
term does not vanish.
\end{remark}

With all these preparations, we now can prove

\begin{theorem}
\label{nonde}The solution $T_{2}:=2\partial_{x}^{2}\iota_{2}$ is
nondegenerate in the following sense: Suppose $\varphi$ is a function
decaying in the $x$ variable and $\frac{2\pi}{kb}$-periodic in the $y$
variable, satisfying
\[
\partial_{x}^{2}\left(  \partial_{x}^{2}\varphi-\varphi+6T_{2}\varphi\right)
-\partial_{y}^{2}\varphi=0.
\]
Then $\varphi=c_{1}\partial_{x}T_{2}+c_{2}\partial_{y}T_{2}.$
\end{theorem}

\begin{proof}
Let $K:=\iota_{2}\partial_{x}^{-2}\frac{\varphi}{2}.$ Then $K$ can be written
as $K=c_{1}\iota_{2}+c_{2}\iota_{2}x+K^{\ast},$ with $\left\vert K^{\ast
}\right\vert \leq C.$

By Lemma \ref{Lemma1} and Lemma \ref{Lemma2}, we can solve the linearized
B\"{a}cklund transformation between $\iota_{0},\iota_{1},\iota_{2},$ and get a
kernel $\xi,$ given by Lemma \ref{Lemma2}, for the linearized bilinear KP-I
operator at $\iota_{0},$ with suitable growth and decay estimate. $\xi$
satisfies
\begin{equation}
\partial_{x}^{4}\xi-\partial_{x}^{2}\xi-\partial_{y}^{2}\xi=0. \label{xi}%
\end{equation}
Since $\xi$ is $\frac{4\pi}{kb}$-periodic in $y$, we conclude that in terms of
the Fourier series expansion,
\[
\xi\left(  x,y\right)  =c_{0}+\sum\limits_{n=1}^{+\infty}\left(  c_{n}\left(
x\right)  \cos\left(  \frac{kb}{2}ny\right)  +d_{n}\left(  x\right)
\sin\left(  \frac{kb}{2}ny\right)  \right)  .
\]
Inserting this into $\left(  \ref{xi}\right)  ,$ we infer that $c_{n}$ and
$d_{n}$ behave like $\exp\left(  \gamma_{n}x\right)  ,$ as $x\rightarrow
+\infty.$ Here $\gamma_{n}$ is determined by the equation
\begin{equation}
\gamma_{n}^{4}-\gamma_{n}^{2}+\left(  \frac{kb}{2}n\right)  ^{2}=0.
\label{gamma}%
\end{equation}
The estimate $\left(  \ref{zi}\right)  $ tells us that the main order of $\xi$
is
\[
C_{1}\left(  y\right)  e^{\frac{3}{2}kx}-C_{2}e^{\frac{k-3\mu}{2}x+\frac
{\sqrt{3}\left(  1-2k^{2}\right)  }{4}iy}.
\]
This together with $\left(  \ref{gamma}\right)  $ tells us that indeed
$C_{1}\left(  y\right)  =C_{2}=0.$ Now using the estimate $\left(
\ref{negative}\right)  $ and fact that
\[
\frac{kb}{\sqrt{\gamma^{2}-\gamma^{4}}}\notin\mathbb{N},\text{ for }\gamma
\in\lbrack\frac{k}{2},k),
\]
we conclude that $\xi=\left(  c_{1}\cos\left(  kby\right)  +d_{1}\sin\left(
kby\right)  \right)  e^{kx}.$ To deal with this possibility, we now use Lemma
\ref{NG} to infer that the system $\left(  \ref{reverse}\right)  $ has a
solution $\eta$ with main order of the form
\begin{equation}
\eta\left(  x,y\right)  =a\left(  y\right)  e^{\frac{3}{2}kx},\text{ as
}x\rightarrow+\infty, \label{be}%
\end{equation}
where $a\left(  y\right)  \neq0.$ Then after we solve the reverse linearized
B\"{a}cklund transformation from $\iota_{1}$ to $\iota_{2},$ we can see that
$\left(  \ref{be}\right)  $ contradicts with the asymptotic behavior of the
kernel $K$ of the linearized bilinear KP-I at $\iota_{2}.$ This finishes the proof.
\end{proof}

\section{Morse index and orbital stability of the lump solution\label{Morse}}

Fix a $k\in\left(  0,\frac{1}{2}\right)  .$ Then $\iota_{2}$ is positive. It
is periodic in $y$, with period $t_{k}:=\frac{2\pi}{k\sqrt{1-k^{2}}}.$ We
observe that as $k\rightarrow0,$ $t_{k}\rightarrow+\infty,$ the function
$2\partial_{x}^{2}\ln\iota_{2}$ converges to the lump $Q$. On the other hand,
as $k\rightarrow\frac{1}{2},$
\[
\iota_{2}\rightarrow2\cosh\frac{x}{2}.
\]
Hence $2\partial_{x}^{2}\ln\iota_{2}$ converges to the one dimensional profile
$\frac{1}{2}\cosh^{-2}\left(  \frac{x}{2}\right)  $. Note that as
$k\rightarrow\frac{1}{2},$
\[
t_{k}\rightarrow\frac{8\sqrt{3}\pi}{3}.
\]
Let us consider the linearized operator around this one dimensional solution.

\begin{lemma}
\label{one-dimensional}The operator
\[
\mathbb{L}\eta:=-\partial_{x}^{2}\eta+\eta-3\cosh^{-2}\left(  \frac{x}%
{2}\right)  \eta+\partial_{x}^{-2}\partial_{y}^{2}\eta
\]
has exactly one negative eigenvalue in the space $L^{2}\left(  \mathbb{R}%
\times\mathbb{R}/\frac{8\sqrt{3}\pi}{3}\mathbb{Z}\right)  .$
\end{lemma}

\begin{proof}
This essentially follows from Lemma 2.3 of \cite{F}.

Let $\lambda$ be a negative eigenvalue of , with $\eta$ being an
eigenfunction:%
\[
-\partial_{x}^{2}\eta+\eta-3\cosh^{-2}\left(  \frac{x}{2}\right)
\eta+\partial_{x}^{-2}\partial_{y}^{2}\eta=\lambda\eta.
\]
Since $\eta$ is $\frac{8\sqrt{3}\pi}{3}$-period in $y,$ we can write $\eta$ as
Fourier series:
\[
\eta\left(  x,y\right)  =\sum\limits_{n}\left(  a_{n}\left(  x\right)
\cos\left(  \frac{\sqrt{3}}{4}ny\right)  +b_{n}\left(  x\right)  \sin\left(
\frac{\sqrt{3}}{4}ny\right)  \right)  .
\]
For the $n=0$ component, the operator reduces to
\[
-\partial_{x}^{2}\eta+\eta-3\cosh^{-2}\left(  \frac{x}{2}\right)  \eta.
\]
$\allowbreak$ It is well known that this operator has a unique negative eigenvalue.

For general $n\in\mathbb{N}$, we get an equation of the form
\[
B_{n}a:=-a^{\prime\prime}+a-3\cosh^{-2}\left(  \frac{x}{2}\right)  a-\frac
{3}{16}n^{2}\partial_{x}^{-2}a=\lambda a.
\]
Note that for $n=1,$ differentiating the family of periodic solution at $k=0$
yields an eigenfunction for the eigenvalue $\lambda=0.$ Consider the operator
\[
\mathbb{T}a:=-a^{\left(  4\right)  }+\left(  \left(  1-3\cosh^{-2}\left(
\frac{x}{2}\right)  \right)  a^{\prime}\right)  ^{\prime}-\frac{3}{16}a.
\]
Observe that $\mathbb{T}a=\partial_{x}B\partial_{x}a.$ From Lemma 2.3 of
\cite{F}(see equation 2.5 there, see also \cite{A}), we know that the operator
$\mathbb{T}$ has no positive eigenvalue. Hence $B_{1}$ has no negative
eigenvalue. Consequently, for any $n>1,$ the quadratic form associated to
$B_{n}$ is positive definite. This finishes the proof.
\end{proof}

Let $E$ be the natural energy space associated to the KP-I, with the norm
$\left\Vert \cdot\right\Vert .$%
\[
E:=\left\{  f:\left\Vert f\right\Vert ^{2}:=\int_{\mathbb{R}^{2}}\left(
\left\vert \partial_{x}f\right\vert ^{2}+f^{2}+\left(  \partial_{x}%
^{-1}\partial_{y}f\right)  ^{2}\right)  <+\infty\right\}  .
\]

\begin{theorem}
\label{stability}The lump is orbitally stable in the following sense: For any
$\varepsilon>0,$ there exists $\delta>0,$ such that, if $u\left(
t;x,y\right)  $ is solution of KP-I with $\left\Vert u\left(  0\right)
-Q\right\Vert <\delta,u\left(  0\right)  \in E,$ then for all $t\in\left(
0,+\infty\right)  ,$
\[
\inf_{\gamma_{1},\gamma_{2}\in\mathbb{R}}\left\Vert u\left(  t\right)
-Q\left(  \cdot+\gamma_{1},\cdot+\gamma_{2}\right)  \right\Vert <\varepsilon.
\]

\end{theorem}

\begin{proof}
We first claim that the Morse index of the lump solution is equal to one. Here
we define its Morse index to be the number of negative eigenvalues(counted
with multiplicity) of the operator $\mathcal{L}:$
\[
\eta\rightarrow-\partial_{x}^{2}\eta+\eta-6Q\eta+\partial_{x}^{-2}\partial
_{y}^{2}\eta.
\]
We also define the Morse index of the periodic solution $T_{2}$(note that
these solutions depending on $k$) to be the number of negative eigenvalues of
the operator $\mathcal{L}_{k}$
\[
\eta\rightarrow-\partial_{x}^{2}\eta+\eta-6T_{2}\eta+\partial_{x}^{-2}%
\partial_{y}^{2}\eta
\]
in the space of functions which are $t_{k}$-periodic in $y.$ We know that
these negative eigenvalues have minimax characterization, hence depend
continuously on the potential $T_{2}.$ Observe that $T_{2}$ depends
continuously $k$ and as $k\rightarrow0,$ it converges to the lump on any
compact set. Now Theorem \ref{nonde} tells us that $T_{2}$ is nondegenerate,
in the sense that the kernels of $\mathcal{L}_{k}$ are generated by space
translation in the $x$ and $y$ directions. This implies that the Morse index
of $T_{2}$ is invariant along this family of periodic solutions $T_{2}$. By
Lemma \ref{one-dimensional}, the Morse index of the 1D solution is one.
Therefore the Morse index of $T_{2}$ is equal to one. It then follows that the
Morse index of the lump solution is also equal to one.

With the spectral property of the lump solution understood, its orbital
stability essentially follows from a result of
Grillakis-Shatah-Strauss \cite{GSS}. (See section 5 of \cite{Strauss}, also see
\cite{F} for the case of line solitons.) We briefly sketch the proof below.

Let $u_{c}$ be the family of lumps with speed $c.$ Define the function
\[
d\left(  c\right)  :=\int_{\mathbb{R}^{2}}\left(  \frac{1}{2}\left(
\partial_{x}u_{c}\right)  ^{2}-u_{c}^{3}+\frac{1}{2}\left(  \partial
_{y}\partial_{x}^{-1}u_{c}\right)  ^{2}+\frac{1}{2}cu_{c}^{2}\right)  .
\]
Then $d^{\prime}\left(  c\right)  =\frac{1}{2}\sqrt{c}\int_{\mathbb{R}^{2}%
}u_{1}^{2}.$ Hence
\begin{equation}
d^{\prime\prime}\left(  c\right)  >0. \label{d}%
\end{equation}
Now using $\left(  \ref{d}\right)  $ and arguing similarly as Lemma 5.1 of
\cite{Strauss}, we can prove that if $\rho$ is a function orthogonal in
$L^{2}$ to $Q$ and $\partial_{x}Q,\partial_{y}Q,$ then,
\[
\int_{\mathbb{R}^{2}}\rho\left(  \mathcal{L}\rho\right)  \geq c\left\Vert
\rho\right\Vert ^{2},
\]
for some constant $c>0.$ We define
\[
Q_{a,b}:=Q\left(  x-a,y-b\right)  .
\]
Let $u\left(  t;x,y\right)  $ be a solution of the KP-I flow, which is
initially close to the lump:
\[
\left\Vert u\left(  0\right)  -Q_{0,0}\right\Vert <\delta.
\]
Without loss of generality, we assume $\left\Vert u\left(  0\right)
\right\Vert _{L^{2}}=\left\Vert Q_{0,0}\right\Vert _{L^{2}}.($The general case
follows from a scaling argument for $Q$ in the travelling speed). The
existence of this solution is guaranteed by the result of \cite{Kenig2}. For
$t$ sufficiently small, we can find $\gamma_{1},\gamma_{2},$ depending on $t,$
such that the function
\[
w\left(  t\right)  :=u\left(  t\right)  -Q_{\gamma_{1},\gamma_{2}}%
\]
satisfies the orthogonality condition
\begin{equation}
\int_{\mathbb{R}^{2}}\left(  w\partial_{x}Q_{\gamma_{1},\gamma_{2}}\right)
=\int_{\mathbb{R}^{2}}\left(  w\partial_{y}Q_{\gamma_{1},\gamma_{2}}\right)
=0. \label{or}%
\end{equation}
On the other hand, $w$ can be written as
\[
w=\alpha Q_{\gamma_{1},\gamma_{2}}+w_{1},
\]
for some small constant $\alpha\left(  t\right)  $ and a function $w_{1}$
satisfying $\int_{\mathbb{R}^{2}}\left(  w_{1}Q_{\gamma_{1},\gamma_{2}%
}\right)  =0.$ Note that by $\left(  \ref{or}\right)  ,$%
\[
\int_{\mathbb{R}^{2}}\left(  w_{1}\partial_{x}Q_{\gamma_{1},\gamma_{2}%
}\right)  =\int_{\mathbb{R}^{2}}\left(  w_{1}\partial_{y}Q_{\gamma_{1}%
,\gamma_{2}}\right)  =0.
\]
Moreover, by the conservation of $L^{2}$ norm, we have%
\[
\int_{\mathbb{R}^{2}}\left(  \left(  1+\alpha\right)  Q_{\gamma_{1},\gamma
_{2}}+w_{1}\right)  ^{2}=\int_{\mathbb{R}^{2}}Q_{0,0}^{2}.
\]
It follows that
\begin{equation}
\left(  2\alpha+\alpha^{2}\right)  \int_{\mathbb{R}^{2}}Q_{0,0}^{2}%
=-\int_{\mathbb{R}^{2}}w_{1}^{2}. \label{alpha}%
\end{equation}
On the other hand, the Hamiltonian energy
\[
H\left(  u\right)  =\int_{\mathbb{R}^{2}}\left(  \frac{1}{2}\left(
\partial_{x}u\right)  ^{2}-u^{3}+\frac{1}{2}\left(  \partial_{y}\partial
_{x}^{-1}u\right)  ^{2}+\frac{1}{2}u^{2}\right)
\]
is conserved by this flow. Hence for any $t,$ $H\left(  u\right)  =H\left(
u\left(  0\right)  \right)  .$ Note that
\[
H(u)=H\left(  Q_{\gamma_{1},\gamma_{2}}\right)  +\int_{\mathbb{R}^{2}}w\left(
\mathcal{L}w\right)  -\int_{\mathbb{R}^{2}}w^{3}.
\]
Anisotropic Sobolev inequality(see for instance equation (2.9) in \cite{F})
tells us that
\[
\int_{\mathbb{R}^{2}}w^{3}\leq C\left\Vert w\right\Vert ^{3}.
\]
Applying the spectral property of the $\mathcal{L}$ and the estimate $\left(
\ref{alpha}\right)  $ of $\alpha,$ we deduce
\[
H\left(  u\right)  \geq H(Q_{\gamma_{1},\gamma_{2}})+C\left\Vert w\right\Vert
^{2}.
\]
In view of the conservation of Hamiltonian $H$ along the KP-I flow, we get
\[
\left\Vert w\right\Vert ^{2}\leq C\left(  H\left(  u\left(  0\right)  \right)
-H\left(  Q_{0,0}\right)  \right)  .
\]
This completes the proof.
\end{proof}

\begin{remark}
Now with the spectral property of the periodic solutions $T_{2}$ understood,
it is also possible to get their orbitally stability in the class of
$\frac{2\pi}{kb}$-periodic solutions(The case of line solitons is discussed in
\cite{F,Yamakazi}). We will not address this issue here.
\end{remark}

\end{document}